\documentclass[11pt,a4paper]{amsart}
\usepackage[all]{xy}
\usepackage{amssymb}

\numberwithin{equation}{section}
\newcommand{\ie}{{\em i.e.}\ }
\newcommand{\confer}{{\em cf.}\ }

\newcommand{\ko}{\: , \;}
\newcommand{\ul}[1]{\underline{#1}}

\newtheorem{theorem}{Theorem}[section]
\newtheorem*{theorem*}{Theorem}

\newtheorem{lemma}[theorem]{Lemma}
\newtheorem{proposition}[theorem]{Proposition}
\newtheorem{corollary}[theorem]{Corollary}
\newtheorem{conjecture}[theorem]{Conjecture}
\newtheorem*{conjecture*}{Conjecture}

\newtheorem{example}[theorem]{Example}
\newtheorem{remark}[theorem]{Remark}
\newtheorem{definition}[theorem]{Definition}

\newcommand{\opname}[1]{\operatorname{\mathsf{#1}}}

\renewcommand{\mod}{\opname{mod}\nolimits}

\newcommand{\Mod}{\opname{Mod}\nolimits}

\newcommand{\per}{\opname{per}\nolimits}
\newcommand{\add}{\opname{add}\nolimits}
\newcommand{\op}{^{op}}

\newcommand{\dimv}{\underline{\dim}\,}

\newcommand{\rank}{\opname{rank}\nolimits}

\newcommand{\ind}{\opname{ind}}
\newcommand{\coind}{\opname{coind}}

\renewcommand{\ker}{\opname{ker}\nolimits}
\newcommand{\obj}{\opname{obj}\nolimits}

\newcommand{\Z}{\mathbb{Z}}

\newcommand{\Q}{\mathbb{Q}}
\newcommand{\C}{\mathbb{C}}

\renewcommand{\P}{\mathbb{P}}
\newcommand{\T}{\mathbb{T}}

\newcommand{\iso}{\stackrel{_\sim}{\rightarrow}}

\newcommand{\sgn}{\opname{sgn}}

\newcommand{\Hom}{\opname{Hom}}

\newcommand{\Ext}{\opname{Ext}}

\newcommand{\End}{\opname{End}}

\newcommand{\ca}{{\mathcal A}}

\newcommand{\cc}{{\mathcal C}}
\newcommand{\cd}{{\mathcal D}}
\newcommand{\ce}{{\mathcal E}}
\newcommand{\cf}{{\mathcal F}}

\newcommand{\ch}{{\mathcal H}}

\newcommand{\cp}{{\mathcal P}}
\newcommand{\cq}{{\mathcal Q}}

\newcommand{\ct}{{\mathcal T}}
\newcommand{\cu}{{\mathcal U}}

\renewcommand{\hat}[1]{\widehat{#1}}

\begin{document}

\author{Changjian Fu and  Bernhard Keller}
\title{On cluster algebras with coefficients and 2-Calabi-Yau categories}
\date{October 16, 2007. Last modified on December 28, 2008}
\address{Department of Mathematics\\
Sichuan University\\
610064 Chengdu\\
P.R.China
}
\email{
\begin{minipage}[t]{5cm}
flyinudream@yahoo.com.cn
\end{minipage}
}
\address{
U.F.R. de Math\'ematiques\\
Institut de Math\'ematiques\\
U.M.R. 7586 du CNRS\\
Universit\'e Denis Diderot - Paris 7\\
2, place Jussieu\\
75251 Paris Cedex 05\\
France
}
\email{
\begin{minipage}[t]{5cm}
keller@math.jussieu.fr \\
www.math.jussieu.fr/ $\widetilde{ }$ keller
\end{minipage}
}

\begin{abstract}
  Building on work by Geiss-Leclerc-Schr\"oer and by
  Buan-Iyama-Reiten-Scott we investigate the link between certain cluster
  algebras with coefficients and suitable 2-Calabi-Yau categories.
  These include the cluster-categories associated with acyclic quivers
  and certain Frobenius subcategories of module categories over
  preprojective algebras. Our motivation comes from the conjectures
  formulated by Fomin and Zelevinsky in `Cluster algebras IV:
  Coefficients'. We provide new evidence for Conjectures~5.4, 6.10,
  7.2, 7.10 and 7.12 and show by an example that the statement of
  Conjecture~7.17 does not always hold.
\end{abstract}

\maketitle

\section{Introduction}
In this article, we pursue the representation-theoretic approach to
Fomin-Zelevinsky's cluster algebras \cite{FominZelevinsky02}
\cite{FominZelevinsky03} \cite{BerensteinFominZelevinsky05}
\cite{FominZelevinsky07} developed by Marsh-Reineke-Zelevinsky
\cite{MarshReinekeZelevinsky03}, Buan-Marsh-Reineke-Reiten-Todorov
\cite{BuanMarshReinekeReitenTodorov06}
\cite{BuanMarshReitenTodorov07}, Caldero-Chapoton and Caldero-Keller
\cite{CalderoChapoton06} \cite{CalderoKeller06}
\cite{CalderoKeller08}, Geiss-Leclerc-Schr\"oer
\cite{GeissLeclercSchroeer05} \cite{GeissLeclercSchroeer06}
\cite{GeissLeclercSchroeer07b}, \cite{BuanIyamaReitenScott07} and
many others, \confer the surveys \cite{BuanMarsh06} \cite{Keller08c}
\cite{Reiten06} \cite{Ringel07} .

Our emphasis here is on cluster algebras {\em with coefficients}.
More precisely, we investigate certain symmetric cluster algebras of
geometric type with coefficients. Coefficients are of great
importance both in geometric examples of cluster algebras
\cite{GekhtmanShapiroVainshtein03}
\cite{GekhtmanShapiroVainshtein05}
\cite{BerensteinFominZelevinsky05}
\cite{Scott06} \cite{GeissLeclercSchroeer07b} and in the study of
duality phenomena \cite{FockGoncharov03} as shown in
\cite{FominZelevinsky07}. Following \cite{BuanIyamaReitenScott07},
we consider two types of categories which allow us to incorporate
coefficients into the representation-theoretic model:
\begin{itemize}
\item[1)] $2$-Calabi-Yau Frobenius categories;
\item[2)] $2$-Calabi-Yau `subtriangulated' categories, {\em i.e.} full
  subcategories of the form $^\perp\!(\Sigma {\mathcal D})$ of a
  $2$-Calabi-Yau triangulated category ${\mathcal C}$, where
  ${\mathcal D}$ is a rigid functorially finite subcategory of
  ${\mathcal C}$ and $\Sigma$ the suspension functor of ${\mathcal
    C}$.
\end{itemize}
In both cases, we establish the link between the category and its associated
cluster algebra using (variants of) cluster characters in the sense of
Palu \cite{Palu06}.  For subtriangulated categories, we use the
restriction of the cluster characters constructed in \cite{Palu06}.
For Frobenius categories, we construct a suitable variant in section~\ref{CC2FC}
(Theorem~\ref{thm1}).

The work of Geiss-Leclerc-Schr\"oer \cite{GeissLeclercSchroeer06}
\cite{GeissLeclercSchroeer07b} and Buan-Iyama-Reiten-Scott
\cite{BuanIyamaReitenScott07} provides us with large classes of
$2$-Calabi-Yau Frobenius categories and of $2$-Calabi-Yau
subtriangulated categories which admit cluster structures in the
sense of \cite{BuanIyamaReitenScott07}. Our general results imply
that for these classes, the $2$-Calabi-Yau categories do yield
`categorifications' of the corresponding cluster algebras with
coefficients (Theorems~\ref{thm2} and ~\ref{thm3}). As an
application, we show that Conjectures 7.2, 7.10 and 7.12 of
\cite{FominZelevinsky07} hold for these cluster algebras
(Proposition~\ref{prop5} and Theorem~\ref{thm3}). Let us recall the
statements of these conjectures:
\begin{itemize}
\item[7.2] cluster monomials are linearly independent;
\item[7.10] different cluster monomials have different
  $\mathbf{g}$-vectors and the $\mathbf{g}$-vectors of the cluster
  variables in any given cluster form a basis of the ambient lattice;
\item[7.12] the $\mathbf{g}$-vectors of a cluster variable with respect to
two neighbouring clusters are related by a certain piecewise linear
transformation (so that the $\mathbf{g}$-vectors equal the $\mathbf{g}^\dagger$-vectors
of \cite{DehyKeller08}).
\end{itemize}

In the case of cluster algebras with principal coefficients admitting
a categorification by a 2-Calabi-Yau subtriangulated category, we
obtain a representation-theoretic interpretation of the $F$-polynomial
defined in section~3 of \cite{FominZelevinsky07}, \confer Theorem~\ref{thm4}.  This interpretation implies in
particular that Conjecture~5.4 of \cite{FominZelevinsky07} holds in
this case: The constant coefficient of the $F$-polynomial equals $1$.
We also deduce that the multidegree of the $F$-polynomial associated
with a rigid indecomposable equals the dimension vector of the
corresponding module (Proposition~\ref{prop7}). By combining this with recent work by
Buan-Marsh-Reiten \cite{BuanMarshReiten08a}, \confer also \cite{Dupont07}, we obtain a counterexample to
Conjecture~7.17 (and 6.11) of \cite{FominZelevinsky07}.  We point out
that the corresponding computations were already present in
G.~Cerulli's work \cite{Cerulli08}. Following a suggestion by A.~Zelevinsky,
we show that, by assuming the existence of suitable categorifications,
instead of the equality claimed in Conjecture~7.17, one does have an
inequality: The multidegree of the $F$-polynomial is greater or equal
to the denominator vector (Proposition~\ref{prop8}). We also show in certain cases that the transformation rule for
$\mathbf{g}$-vectors predicted by Conjecture~6.10 of
 \cite{FominZelevinsky07} does hold (Proposition~\ref{prop9}).

Let us emphasize that our proofs for certain cluster algebras
of some of the conjectures
of \cite{FominZelevinsky07} depend on the existence of suitable
Hom-finite $2$-Calabi-Yau categories with a cluster-tilting object.
This hypothesis imposes a finiteness condition on the corresponding
cluster algebra (to the best of our knowledge, it is not known how
to express this condition in combinatorial terms). The construction
of such $2$-Calabi-Yau categories is a non trivial problem for which
we rely on \cite{BuanMarshReinekeReitenTodorov06} in the acyclic
case and on \cite{GeissLeclercSchroeer06}
\cite{GeissLeclercSchroeer07b} \cite{BuanIyamaReitenScott07} and
\cite{Amiot08a} in the non acyclic case. As A.~Zelevinsky has kindly
informed us, many of the conjectures of \cite{FominZelevinsky07}
will be proved in \cite{DerksenWeymanZelevinsky09} in full
generality building on \cite{MarshReinekeZelevinsky03} and
\cite{DerksenWeymanZelevinsky07}.

{\bf Acknowledgments.} We thank A.~Zelevinsky for stimulating
discussions and for informing us about the ongoing work in
\cite{DerksenWeymanZelevinsky09} and \cite{Cerulli08}. We are grateful to
A.~Buan and I.~Reiten for sharing the results on dimension vectors
obtained in \cite{BuanMarshReiten08a}. We thank J.~Schr\"oer and O.~Iyama for
their interest and for motivating discussions. The first-named
author gratefully acknowledges a 5-month fellowship of the Liegrits
network (MRTN-CT 2003-505078) during which this research was carried
out. Both authors thank the referee for his great help in making
this article more readable.

\section{Recollections}

\subsection{Cluster algebras} \label{ss:ClusterAlgebras}
In this section, we recall
the construction of cluster algebras of geometric type with
coefficients from \cite{FominZelevinsky07}. For an integer $x$,
we use the notations
\[
[x]_+=\max(x,0)
\]
and
\[
\sgn(x)= \left\{ \begin{array}{ll} -1 & \mbox{if $x<0$;} \\ 0 & \mbox{if $x=0$;} \\ -1 & \mbox{if $x<0$.}
\end{array} \right.
\]

The {\em tropical semifield} on a finite family of variables $u_j$,
$j\in J$, is the abelian group (written multiplicatively) freely
generated by the $u_j$, $j\in J$, endowed with the {\em addition}
$\oplus$ defined by
\[
\prod_j u_j^{a_j} \oplus \prod_j u_j^{b_j} = \prod_j u_j^{\min(a_j,
b_j)}.
\]

Let $1\leq r\leq n$ be integers.  Let $\P$ be the tropical semifield
on the indeterminates $x_{r+1}, \ldots, x_n$. Let $\Q\P$ be the
group algebra on the abelian group $\P$. It identifies with the
algebra of Laurent polynomials with rational coefficients in the
variables $x_{r+1},\ldots , x_n$. Let $\cf$ be the field of
fractions of the ring of polynomials with coefficients in $\Q\P$ in
$r$ indeterminates. A {\em seed} in $\cf$ is a pair
$(\tilde{B},\mathbf{x})$ formed by
\begin{itemize}
\item an $n\times r$-matrix $\tilde{B}$ with integer entries
whose {\em principal part} $B$ (the submatrix formed by the first
$r$ rows) is antisymmetric;
\item a free generating set $\mathbf{x}=\{x_1, \ldots, x_r\}$
of the field $\cf$.
\end{itemize}
The matrix $\tilde{B}$ is called the {\em exchange matrix}
and the set $\mathbf{x}$ the {\em cluster} of the
seed $(\tilde{B}, \mathbf{x})$.
Let $1\leq s\leq r$ be an integer. The {\em seed mutation in the direction $s$}
transforms the seed $(\tilde{B},\mathbf{x})$ into the seed
$\mu_s(\tilde{B},\mathbf{x})=(\tilde{B}', \mathbf{x}')$, where
\begin{itemize}
\item the entries $b'_{ij}$ of $\tilde{B}'$ are given by
\[
b'_{ij} = \left\{ \begin{array}{ll} -b_{ij} & \mbox{if $i=s$ or
$j=s$; } \\
b_{ij}+ \sgn(b_{is}) [b_{is} b_{sj}]_+ & \mbox{otherwise.}
\end{array} \right.
\]
\item The cluster $\mathbf{x'}=\{x'_1, \ldots x'_r\}$ is
given by $x'_j=x_j$ for $j\neq s$ whereas $x'_s\in\cf$ is determined
by the {\em exchange relation}
\[
x'_s x_s = \prod_{i=1}^n x_i^{[b_{is}]_+} + \prod_{i=1}^n
x_i^{[-b_{is}]_+}.
\]
\end{itemize}
Mutation in a fixed direction is an involution.

Let $\T_r$ be the {\em $r$-regular tree}, whose edges are labeled by
the numbers $1,\ldots, r$ so that the $r$ edges emanating from each
vertex carry different labels. A {\em cluster pattern} is the assignment
of a seed $(\tilde{B}_t, \mathbf{x}_t)$ to each vertex $t$ of $\T_r$ such
that the seeds assigned to vertices $t$ and $t'$ linked by
an edge labeled $s$ are obtained from each other by the seed
mutation $\mu_s$.

Fix a vertex $t_0$ of the $r$-regular tree $\T_r$. Clearly, a
cluster pattern is uniquely determined by the {\em initial seed}
$(\tilde{B}_{t_0}, x_{t_0})$, which can be chosen arbitrarily.

Fix a seed $(\tilde{B}, \mathbf{x})$ and let
$(\tilde{B}_t, \mathbf{x}_t)$, $t\in \T_r$
be the unique cluster pattern with initial seed $(\tilde{B}, \mathbf{x})$.
The {\em clusters} associated with $(\tilde{B}, \mathbf{x})$ are
the sets $\mathbf{x}_t$, $t\in\T_r$. The {\em cluster variables}
are the elements of the clusters. The {\em cluster algebra}
$\ca(\tilde{B}) = \ca(\tilde{B},\mathbf{x})$ is the $\Z\P$-subalgebra of $\cf$
generated by the cluster variables. Its {\em ring of coefficients} is
$\Z\P$. It is a `cluster algebra without coefficients' if $r=n$ and
thus $\Z\P=\Z$.

\subsection{Cluster algebras from ice quivers} \label{ss:IceQuivers}
As we have seen in the
previous subsection, our cluster
algebras are given by certain integer matrices $\tilde{B}$. Such
matrices can also be encoded by `ice quivers': A {\em quiver} is a
quadruple $Q=(Q_0, Q_1, s,t)$, where $Q_0$ is a set (the set of {\em
vertices}), $Q_1$ is a set (the set of {\em arrows}) and $s$ and $t$
are two maps $Q_1 \to Q_0$ (taking an arrow to its {\em source}
respectively to its {\em target}). An {\em ice quiver} is a pair
$(Q,F)$ consisting of a quiver $Q$ and a subset $F$ of its set of
vertices ($F$ is the set of {\em frozen vertices}).

Let $(Q,F)$ be an ice quiver such that the set $Q_0$ is the set of natural
numbers from $1$ to $n$, the set $Q_1$ is finite and the set $F$ is the set of
natural numbers from $r+1$ to $n$ for some $1\leq r\leq n$.
The associated integer matrix $\tilde{B}(Q,F)$ is the $n\times r$
matrix whose entry $b_{ij}$ equals the
number of arrows from $i$ to $j$ minus the number of arrows from $j$ to $i$.
The {\em cluster algebra with coefficients $\ca(Q,F)$} is defined as the cluster
algebra $\ca(\tilde{B}(Q,F))$.
The matrix $\tilde{B}(Q,F)$ determines the ice quiver $(Q,F)$ if
\begin{itemize}
\item[1)]  $Q$ does not have loops (arrows from a vertex to itself) and
\item[2)] $Q$ does not have $2$-cycles (pairs of distinct arrows $\alpha$, $\beta$ such that
$s(\alpha)=t(\beta)$ and $t(\alpha)=s(\beta)$) and
\item[3)] there are no arrows between any vertices of $F$.
\end{itemize}
Given integers $1 \leq r\leq n$ each integer matrix $\tilde{B}$ with
antisymmetric principal part $B$ (formed by the first $r$ rows of
$\tilde{B})$, is obtained as the matrix associated with a unique ice
quiver satisfying these properties. The {\em mutation of
ice quivers} satisfying conditions 1)-3) is defined via
the mutation of the corresponding integer matrices recalled
in section~\ref{ss:ClusterAlgebras}.

\subsection{Krull-Schmidt categories} \label{ss:KrullSchmidtCategories}
An additive category {\em has split idempotents}
if each idempotent endomorphism $e$ of an object $X$  gives rise to a direct sum decomposition
$Y\oplus Z\iso X$ such that $Y$ is a kernel for $e$. A {\em Krull-Schmidt category} is an
additive category where the endomorphism rings of indecomposable objects are
local and each object decomposes into a finite direct sum of indecomposable
objects (which are then unique up to isomorphism and permutation).
Each Krull-Schmidt category has split idempotents.
We {\em write $\mathsf{indec}(\cc)$} for the set of isomorphism classes
of indecomposable objects of a Krull-Schmidt category $\cc$.

Let $\cc$ be a Krull-Schmidt category.
An object $X$ of $\cc$ is {\em basic} if every indecomposable of $\cc$
occurs with multiplicity $\leq 1$ in $X$. In this case, $X$ is fully
determined by the full additive {\em subcategory $\add(X)$} whose objects are the
direct factors of finite direct sums of copies of $X$. The map $X \mapsto \add(X)$
yields a bijection between the isomorphism classes of basic objects and
the full additive subcategories of $\cc$ which are stable under
taking direct factors and only contain finitely many indecomposables
up to ismorphism.

Let $k$ be an algebraically closed field. A {\em $k$-category} is
a category whose morphism sets are endowed with structures of
$k$-vector spaces such that the composition maps are bilinear. A
$k$-category is {\em $\Hom$-finite} if all of its morphism spaces
are finite-dimensional. A {\em $k$-linear category} is a $k$-category
which is additive. Let $\cc$ be a $k$-linear $\Hom$-finite
category with split idempotents. Then $\cc$ is a Krull-Schmidt category.
Let $\ct$ be an additive subcategory of $\cc$ stable under taking direct
factors. The {\em quiver $Q=Q(\ct)$ of $\ct$}
is defined as follows: The vertices of $Q$ are the isomorphism classes of indecomposable
objects of $\ct$ and the number of arrows from the isoclass of
$T_1$ to that of $T_2$ equals the dimension of the space of irreducible
morphisms
\[
\mathsf{irr}(T_1, T_2) = \mathsf{rad}(T_1, T_2)/\mathsf{rad}^2(T_1,T_2) \ko
\]
where $\mathsf{rad}$ denotes the {\em radical of $\ct$}, \ie the ideal such
that $\mathsf{rad}(T_1,T_2)$ is formed by all non isomorphisms from $T_1$ to $T_2$.

\subsection{$2$-Calabi-Yau triangulated categories} \label{ss:CalabiYauCategories}
Let $k$ be an algebraically closed field. Let $\cc$ be a $k$-linear
triangulated category with suspension functor $\Sigma$. We assume
that $\cc$ is {\em $\Hom$-finite} and has split idempotents.
Thus, it is a Krull-Schmidt category. For objects
$X$, $Y$ of $\cc$ and an integer $i$, we define
\[
\Ext^i(X,Y)=\cc(X,\Sigma^i Y).
\]
An object $X$ of $\cc$ is {\em rigid} if $\Ext^1(X,X)=0$.

Let $d$ be an integer. Following \cite{VandenBergh08}, \confer~also
\cite{Keller08b}, we define the category $\cc$ to be {\em
$d$-Calabi-Yau} if there exists a family of linear forms
\[
t_X : \cc(X,\Sigma^d X) \to k \ko X\in\obj(\cc)\ko
\]
such that the bilinear forms
\[
\langle, \rangle: \cc(Y, \Sigma^d X) \times \cc(X,Y) \to k \ko (f,g) \mapsto t_X(f\circ g)
\]
are non degenerate and satisfy
\[
\langle \Sigma^p f, g\rangle = (-1)^{pq} \langle \Sigma^q g, f\rangle
\]
for all $f$ in $\cc(Y, \Sigma^q X)$ and all $g\in \cc(X, \Sigma^p Y)$,
where $p+q=d$.

Let us assume that $\cc$ is $2$-Calabi-Yau. A {\em cluster-tilting
subcategory of $\cc$} is a full additive subcategory $\ct\subset
\cc$ which is stable under taking direct factors and such that
\begin{itemize}
\item for each object $X$ of $\cc$, the functors
$\cc(X,?):\ct\to \mod k$ and $\cc(?,X):\ct\op\to\mod k$
are finitely generated;
\item an object $X$ of $\cc$ belongs to $\ct$ iff we have
$\Ext^1(T,X)=0$ for all objects $T$ of $\ct$.
\end{itemize}

A {\em cluster-tilting object} is a basic object $T$ of $\cc$ such that
$\add(T)$ is a cluster-tilting subcategory. Equivalently, an object $T$
is cluster-tilting if it is rigid and if each object $X$ satisfying
$\Ext^1(T,X)=0$ belongs to $\add(T)$. The following definition is
taken from section~1 of \cite{BuanIyamaReitenScott07}. Recall that $\cc$ is
a $\Hom$-finite $k$-linear triangulated category with split
idempotents which is $2$-Calabi-Yau.

\begin{definition}[\cite{BuanIyamaReitenScott07}]
\label{def:TriangClusterStructure}
The cluster-tilting subcategories of $\cc$ {\em determine a cluster
structure on $\cc$} if the following hold:
\begin{itemize}
\item[0)] There is at least one cluster-tilting subcategory in $\cc$.
\item[1)] For each cluster-tilting subcategory $\ct'$ of
$\cc$ and each indecomposable $M$ of $\ct'$, there is a unique (up
to isomorphism) indecomposable $M^*$ not isomorphic to $M$ and such
that the additive subcategory $\ct''=\mu_M(\ct')$ of $\cc$ with
set of indecomposables
\[
\mathsf{indec}(\ct'')=\mathsf{indec}(\ct')\setminus\{M\} \cup
\{M^*\}
\]
is a cluster-tilting subcategory;
\item[2)] In the situation of 1), there are triangles
\[
\xymatrix{M^* \ar[r]^f & E \ar[r]^g & M \ar[r] & \Sigma M^*}
 \mbox{ and }
\xymatrix{M \ar[r]^s &  E' \ar[r]^t &  M^* \ar[r] &  \Sigma M^*} \ko
\]
where $g$ and $t$ are minimal right $\ct'\cap\ct''$-approximations
and $f$ and $s$ are minimal left $\ct'\cap\ct''$-approximations.
\item[3)] For any cluster-tilting subcategory $\ct'$, the quiver
$Q(\ct')$ does not have loops nor $2$-cycles.
\item[4)] We have $Q(\mu_M(\ct'))=\mu_M(Q(\ct'))$ for each cluster-tilting
subcategory $\ct'$ of $\cc$ and each indecomposable $M$ of $\ct'$.
\end{itemize}
The cluster tilting subcategory $\ct''=\mu_M(\ct')$ of 1) is
the {\em mutation of $\ct'$ at the indecomposable object $M$}.
The {\em mutation of a cluster-tilting object $T$} is defined via
the mutation of the cluster-tilting subcategory $\add(T)$.
\end{definition}

\begin{lemma} \label{lemma:Multiplicities} Suppose that the
  cluster-tilting subcategories determine a cluster structure on
  $\cc$. Then, in the situation of condition~2) of the above
  definition~\ref{def:TriangClusterStructure}, the following hold:
\begin{itemize}
\item[a)] The space $\Ext^1(M,M^*)$ is one-dimensional (hence, by the $2$-Calabi-Yau
property, so is the space $\Ext^1(M^*,M)$) and the triangles of 2) are non split.
\item[3)] The multiplicity
of an indecomposable $U$ of $\ct'\cap\ct''$ in $E$ equals the number
of arrows from $U$ to $M$ in the quiver $Q(\ct')$  and that from $M^*$ to $U$
in $Q(\ct'')$; the multiplicity of $U$ in $E'$ equals the number of
arrows from $M$ to $U$ in $Q(\ct')$ and that from $U$ to $M^*$ in
$Q(\ct'')$;
\end{itemize}
\end{lemma}

\begin{proof} a) The first triangle yields an exact sequence
\[
\cc(M,E) \to \cc(M,M) \to \Ext^1(M,M^*) \to 0.
\]
By the absence of loops required in condition 3),
each radical morphism from $M$ to $M$ factors through $E$. Since $k$
is algebraically closed, the radical is of codimension $1$ in the
local algebra $\cc(M,M)$. Thus, the space $\Ext^1(M,M^*)$ is one-dimensional.
The minimality of the approximations implies that the triangles are non split.
b) This follows if we combine the definition of the quivers $Q(\ct')$ and $Q(\ct'')$,
with the approximation properties of $f$, $g$, $s$ and $t$.
\end{proof}

We refer to section~1, page~11 of
\cite{BuanIyamaReitenScott07} for a list of classes of examples
where this assumption holds. In particular, this list contains
the cluster categories associated with finite quivers without oriented
cycles and the stable categories of preprojective algebras of Dynkin quivers.
We refer the reader to the surveys \cite{BuanMarsh06} \cite{Reiten06} \cite{Keller08}
\cite{Keller08c}
for more information on cluster categories and to
the survey \cite{GeissLeclercSchroeer08}
for more information on stable categories of Dynkin quivers.

\subsection{Cluster characters} \label{ss:ClusterCharacters}
The notion of cluster character is due to Palu \cite{Palu08}.
Under suitable assumptions, cluster characters allow one to pass
from $2$-Calabi-Yau categories to cluster algebras. We recall
the definition and the main construction from \cite{Palu08}.
Let $k$ be an algebraically closed field and
$\cc$ a $k$-linear $\Hom$-finite triangulated category with
split idempotents which is $2$-Calabi-Yau.
Let $R$ be a commutative ring. A {\em cluster character} on $\cc$
with values in $R$ is a map $\zeta: \obj(\cc) \to R$ such that
\begin{itemize}
\item[1)] we have $\zeta(L)=\zeta(L')$ if and $L$ and $L'$ are isomorphic,
\item[2)] we have $\zeta(L\oplus M) = \zeta(L)\zeta(M)$ for all objects $L$ and $M$ and
\item[3)] if $L$ and $M$ are objects such that $\Ext^1(L,M)$ is one-dimensional
(and hence $\Ext^1(M,L)$ is one-dimensional) and
\[
L \to E \to M \to \Sigma L \mbox{ and } M \to E' \to L \to \Sigma M
\]
are non-split triangles, then we have
\begin{equation} \label{eq:ClusterCharacterExchangeEquation}
\zeta(L) \zeta(M) = \zeta(E) + \zeta(E').
\end{equation}
\end{itemize}
Assume that $\cc$ has a cluster-tilting object $T$ which is the
direct sum of $r$ pairwise non isomorphic indecomposable summands $T_1$,
\ldots $T_r$. In a vast generalization of Caldero-Chapoton's work
\cite{CalderoChapoton06}, Palu has shown in \cite{Palu08} that there
is a canonical cluster-character
\[
X^T_?: \obj(\cc) \to \Z[x_1, \ldots, x_r] \ko M \mapsto X^T_M
\]
such that $X^T_{\Sigma T_i}=x_i$ for $1\leq i \leq r$. Let us recall
Palu's construction. First we need to introduce some more notation.
Let $C$ be the endomorphism algebra of $T$. Let $\mod C$ denote the
category of $k$-finite-dimensional right $C$-modules. For each
$1\leq i \leq r$, the morphism space $\cc(T,T_i)$ becomes an
indecomposable projective right $C$-module denoted by $P_i$. Its
simple top is denoted by $S_i$. For $L$ and $M$ in $\mod C$, define
\[
\langle L, M\rangle = \dim \Hom_C(L,M) - \dim \Ext^1_C(L,M)
\]
and put
\[
\langle L, M\rangle_a = \langle L,M \rangle - \langle M,L \rangle.
\]
By Theorem~11 of \cite{Palu08}, the map $(L,M) \mapsto \langle
L,M\rangle_a$ induces a well-defined bilinear form on the
Grothendieck group $K_0(\mod C)$. By \cite{KellerReiten07}, for any
$X\in \cc$, we have triangles
\[
 T^X_1\to T^X_0\to X\to \Sigma T^X_1 \mbox{ and }
 X\to \Sigma^2 T^0_X\to \Sigma^2 T^1_X\to \Sigma X,
\]
where $T^X_1$, $T^X_0$, $T_X^0$ and $T_X^1$ belong to $\add(T)$.
The {\em index} and {\em coindex} of $X$ with respect to $T$ are defined to
be  the classes in $K_0(\add T)$
\[
 \ind_T (X)=[T^X_0]-[T^X_1] \mbox { and }
 \coind_T (X)= [T_X^0]-[T_X^1].
\]

For an object $M$ of $\cc$, one defines
\[
X^T_M = \prod_{i=1}^r x_i^{-[\operatorname{coind}_{T}(M): T_i]}
\sum_e \chi(Gr_e(\cc(T,M)) \, \prod_{i=1}^r x_i^{\langle S_i,e\rangle_a},
\]
where $e$ runs through the positive elements of the Grothendieck
group $K_0(\mod C)$ and $Gr_e(\cc(T,M))$ denotes the variety
of submodules $U$ of the right $C$-module $\cc(T,M)$ such
that the class of $U$ is $e$ and $\chi$ is the Euler characteristic
(of the underlying topological space if $k=\C$ or of $l$-adic
cohomology if $k$ is arbitrary).

\subsection{From $2$-CY categories to cluster algebras without coefficients}
In this section, we show how certain $2$-Calabi-Yau triangulated categories can be linked to
cluster algebras without coefficients via cluster characters. All we need to do is to combine the results
recalled in sections \ref{ss:CalabiYauCategories} and \ref{ss:ClusterCharacters}. In
the main part of the paper, we will concentrate on the case where our
cluster algebras do have coefficients.

Let $k$ be an algebraically closed field and $\cc$ a $\Hom$-finite
$k$-linear $2$-Calabi-Yau triangulated category with split idempotents as defined
in section~\ref{ss:CalabiYauCategories}. Let $T$ be a
cluster-tilting object in $\cc$. Assume that $T$ is the direct sum
of $r$ pairwise non isomorphic indecomposable objects $T_1$, \ldots,
$T_r$.  Let
\[
\zeta_T : \obj(\cc) \to \Q(x_1, \ldots, x_r)
\]
be Palu's cluster character associated with $T$ as recalled in
section~\ref{ss:ClusterCharacters}. In particular, we have
\begin{equation} \label{eq:ClusterCharInitialValues}
\zeta_T(\Sigma T_i)=x_i \mbox{ for } 1\leq i\leq r.
\end{equation}
Now assume that the cluster-tilting subcategories define a cluster structure
on $\cc$ (\confer section~\ref{ss:CalabiYauCategories}). A cluster-tilting
object $T'$ is {\em reachable from $T$} if $\add(T')$ is obtained
from $\add(T)$ be a finite sequence of mutations as defined in
\ref{ss:CalabiYauCategories}. A rigid object $M$ is {\em reachable
from $T$} if it lies in $\add(T')$ for a cluster-tilting object $T'$
reachable from $T$. Let $Q$ be the quiver of the endomorphism
algebra $C$ of $T$, or equivalently, the quiver of the category
$\add(T)$. We consider $Q$ as an ice quiver with empty set of frozen
vertices and denote by $\ca(Q)$ the associated cluster algebra
without coefficients (defined by specializing the construction of
\ref{ss:IceQuivers} to the case where the set of frozen vertices is
empty). It is the subalgebra of $\Q(x_1, \ldots, x_r)$ generated by
the cluster variables.

\begin{proposition} \label{prop:ClusterCharacterProperties}
Assume that the above assumptions hold and
in particular that the cluster-tilting subcategories define
a cluster-structure on $\cc$ (\confer
section~\ref{ss:CalabiYauCategories}).
\begin{itemize}
\item[a)] The map $M \mapsto \zeta_T(\Sigma M)$ induces a surjection from
the set of rigid objects reachable from $T$ onto the set of cluster
variables of the cluster algebra $\ca(Q)$.
\item[b)] The surjection of a) induces a
surjection from the set of cluster-tilting objects reachable from
$T$ onto the set of clusters of $\ca(Q)$.
\end{itemize}
\end{proposition}

\begin{proof} Clearly, part a) follows from part b). Let us prove part b).
Let $\T_r$ be the $r$-regular tree and let $t_0$ be a fixed vertex of
$\T_r$. Let $B$ be the antisymmetric matrix associated with the
quiver $Q$ and let $\mathbf{x}$ be the initial cluster $x_1, \ldots, x_r$.
Let $(B_t, \mathbf{x}_t)$, $t\in \T_r$, be the unique cluster pattern
with initial seed $(B_{t_0}, \mathbf{x}_{t_0})=(B,\mathbf{x})$
(\confer section~\ref{ss:ClusterAlgebras}). To each
vertex $t$ of $\T$, we assign a cluster-tilting object $T_t$ with
indecomposable direct summands $T_{t,1}$, \ldots, $T_{t,r}$ such that
\begin{itemize}
\item[1)] we have $T_{t_0}=T$ and
\item[2)] if $t$ is linked to $t'$ by an edge labeled $s$, then $T_{t'}$
is obtained from $T_t$ by mutating at the summand $T_{t,s}$.
\end{itemize}
It follows from point 1) of the definition of a cluster structure that
$T_t$ is well-defined for each vertex $t$ of $\T$. Moreover, it follows
from point 4) of the same definition that for each vertex $t$ of $\T$,
the antisymmetric matrix $B_t$ corresponds to the quiver of the category
$\add(T_t)$ under the bijection of section~\ref{ss:IceQuivers}. We
claim that for each vertex $t$ of $\T$, the cluster character takes
the shift $\Sigma T_{t,i}$ of the indecomposable direct summand $T_{t,i}$
of $T_t$ to the cluster variable $x_{t,i}$, $1 \leq i\leq r$. Indeed, this
holds for $t=t_0$ by equation~(\ref{eq:ClusterCharInitialValues}).
Now assume that it holds for some vertex $t$ and that $t$ is linked
to a vertex $t'$ by an edge labeled $s$. We know that the indecomposable summands
of $T_{t'}$ are the $T_{t',i}=T_{t,i}$ for $i\neq s$ and a new summand $T'_{t,s}$
which is not isomorphic to $T_{t,s}$. By part a) of lemma~\ref{lemma:Multiplicities},
the extension space between $T_{t,s}$ and $T_{t',s}$
is one-dimensional and we have non split triangles
\[
T_{t',s} \to E \to T_{t,s} \to \Sigma T_{t',s} \mbox{ and }
T_{t,s} \to E' \to T_{t',s} \to \Sigma T_{t,s}.
\]
Here, the middle terms are sums of copies of the $T_{t,i}$, $i\neq s$,
and the multiplicities are determined by the quivers of the
endomorphism algebras of $T$ and $T'$, as indicated in part b)
of lemma~\ref{lemma:Multiplicities}. More precisely, if $b_{ij}^t$ denotes the $(i,j)$-entry of
the exchange matrix, the summand $T_{t,i}$ occurs in $E$ with multiplicity
$[b_{is}^t]_+$ and in $E'$ with multiplicity $[b_{si}^t]_+=[-b_{is}^t]_+$. Now
if we use points 2) and 3) of the definition
of a cluster character, we see that the induction hypothesis
and equation~(\ref{eq:ClusterCharacterExchangeEquation})
yield the exchange relation
\[
x_{t,s} \zeta_T(\Sigma T_{t',s}) =\prod_{i=1}^n x_i^{[b^t_{ik}]_+} + \prod_{i=1}^n
x_i^{[-b^t_{ik}]_+}.
\]
Thus, we have $\zeta_T(\Sigma T_{t',s}) = x_{t',s}$ as required.
\end{proof}

\subsection{Frobenius categories} \label{ss:FrobeniusCategories}
A {\em Frobenius category} is an exact category
in the sense of Quillen \cite{Quillen73} which has enough projectives and enough injectives
and where an object is projective iff it is injective. By definition,
such a category is endowed with a class of admissible exact sequences
\[
0 \to L \to M \to N \to 0.
\]
Following \cite{GabrielRoiter92} we will call the left morphism
$L\to M$ of such a sequence an {\em inflation}, the right morphism a
{\em deflation} and, sometimes, the whole sequence a {\em
conflation}. Let $\ce$ be a Frobenius category. Its associated
stable category $\ul{\ce}$ is the quotient of $\ce$ by the ideal of
morphisms factoring through  a projective-injective object. It was
shown by Happel \cite{Happel87} that $\ul{\ce}$ has a canonical
structure of triangulated category. We have
\[
\Ext^i_\ce(L,M) \iso \Ext^i_{\ul{\ce}}(L,M)
\]
for all objects $L$ and $M$ of $\ce$ and all integers $i\geq 1$.
An object $M$ of $\ce$ is {\em rigid} if $\Ext^1_\ce(M,M)=0$.

Let $k$ be an algebraically closed field and $\ce$ a $\Hom$-finite
Frobenius category with split idempotents.
Suppose that $\ce$ is a {\em $2$-Calabi-Yau Frobenius category}, \ie its associated stable
category $\cc=\ul{\ce}$ is $2$-Calabi-Yau in the sense of section~\ref{ss:CalabiYauCategories}.
A {\em cluster-tilting subcategory of $\ce$} is a full additive subcategory $\ct\subset
\ce$ which is stable under taking direct factors and such that
\begin{itemize}
\item for each object $X$ of $\ce$, the functors
$\ce(X,?):\ct\to \mod k$ and $\ce(?,X):\ct\op\to\mod k$
are finitely generated;
\item an object $X$ of $\ce$ belongs to $\ct$ iff we have
$\Ext^1_\ce(T,X)=0$ for all objects $T$ of $\ct$.
\end{itemize}
Clearly if these conditions hold, each projective-injective
object of $\ce$ belongs to $\ct$.
A {\em cluster-tilting object} is a basic object $T$ of $\ce$ such that
$\add(T)$ is a cluster-tilting subcategory. Equivalently, an object $T$
is cluster-tilting if it is rigid and if each object $X$ satisfying
$\Ext^1_\ce(T,X)=0$ belongs to $\add(T)$. The following definition
is taken from section~1 of \cite{BuanIyamaReitenScott07}. Recall
that $\ce$ is a $k$-linear $\Hom$-finite Frobenius category with
split idempotents such that the associated stable category
$\cc=\ul{\ce}$ is $2$-Calabi-Yau.

\begin{definition}[\cite{BuanIyamaReitenScott07}]
\label{def:ExactClusterStructure}
The cluster-tilting subcategories of $\ce$ {\em determine a cluster
structure on $\ce$} if the following hold:
\begin{itemize}
\item[0)] There is at least one cluster-tilting subcategory in $\ce$.
\item[1)] For each cluster-tilting subcategory $\ct'$ of $\ce$ and
  each non projective indecomposable $M$ of $\ct'$, there is a unique
  (up to isomorphism) non projective indecomposable $M^*$ not
  isomorphic to $M$ and such that the additive subcategory
  $\ct''=\mu_M(\ct')$ of $\ce$ with set of indecomposables
\[
\mathsf{indec}(\ct'')=\mathsf{indec}(\ct')\setminus\{M\} \cup
\{M^*\}
\]
is a cluster-tilting subcategory;
\item[2)] In the situation of 1), there are conflations
\[
\xymatrix{0 \ar[r] & M^* \ar[r]^f & E   \ar[r]^g & M \ar[r] & 0}
 \mbox{ and }
\xymatrix{0 \ar[r] & M \ar[r]^s   &  E' \ar[r]^t &  M^* \ar[r] &  0} \ko
\]
where $g$ and $t$ are minimal right $\ct'\cap\ct''$-approximations
and $f$ and $s$ are minimal left $\ct'\cap\ct''$-approximations.
\item[3)] For any cluster-tilting subcategory $\ct'$, the quiver
$Q(\ct')$ does not have loops nor $2$-cycles.
\item[4)] We have $Q^\circ(\mu_M(\ct'))=\mu_M(Q^\circ(\ct'))$ for each
  cluster-tilting subcategory $\ct'$ of $\ce$ and each non projective
  indecomposable $M$ of $\ct'$, where $Q^\circ(\ct')$ denotes the
  quiver obtained from $Q(\ct')$ be removing all arrows between
  projective vertices.
\end{itemize}
The cluster tilting subcategory $\ct''=\mu_M(\ct')$ of 1) is
the {\em mutation of $\ct'$ at the non projective indecomposable object $M$}.
The {\em mutation of a cluster-tilting object $T$} is defined via
the mutation of the cluster-tilting subcategory $\add(T)$.
\end{definition}

\begin{lemma} \label{lemma:ExactMultiplicities} Suppose that the
cluster-tilting subcategories determine a cluster structure on
$\ce$. Then, in the situation of condition~2) of the above
definition~\ref{def:ExactClusterStructure}, the following hold:
\begin{itemize}
\item[a)] The space $\Ext^1(M,M^*)$ is one-dimensional (hence, by the $2$-Calabi-Yau
property, so is the space $\Ext^1(M^*,M)$) and the conflations of 2) are non split.
\item[b)] The multiplicity
of an indecomposable $U$ of $\ct'\cap\ct''$ in $E$ equals the number
of arrows from $U$ to $M$ in the quiver $Q(\ct')$  and that from $M^*$ to $U$
in $Q(\ct'')$; the multiplicity of $U$ in $E'$ equals the number of
arrows from $M$ to $U$ in $Q(\ct')$ and that from $U$ to $M^*$ in
$Q(\ct'')$;
\end{itemize}
\end{lemma}

We omit the proof of the lemma since it is entirely parallel to that
of lemma~\ref{lemma:Multiplicities}. Large classes of examples of
Frobenius categories where the cluster-tilting objects define a
cluster-structure are obtained in \cite{GeissLeclercSchroeer07b} and
\cite{BuanIyamaReitenSmith08}, \confer~the survey
\cite{GeissLeclercSchroeer08} and
example~\ref{example:Frobenius2CYRealization} below. For an
extension of the theory from the antisymmetric to the
antisymmetrizable case, we refer to \cite{Demonet08}.

\section{Cluster characters for 2-Calabi-Yau Frobenius categories}
\label{CC2FC}

Let $k$ be an algebraically closed field and $\ce$ a $k$-linear
Frobenius category with split idempotents. We assume that $\ce$ is
$\Hom$-finite and that the stable category $\cc=\ul{\ce}$ is $2$-Calabi-Yau
(\confer section~\ref{ss:CalabiYauCategories}).

\begin{definition} \label{def:ExactClusterCharacter}
A {\em cluster character on $\ce$
with values in a commutative ring $R$} is a map $\zeta: \obj(\ce) \to R$ such that
\begin{itemize}
\item[1)] we have $\zeta(L)=\zeta(L')$ if and $L$ and $L'$ are isomorphic,
\item[2)] we have $\zeta(L\oplus M) = \zeta(L)\zeta(M)$ for all objects $L$ and $M$ and
\item[3)] if $L$ and $M$ are objects such that $\Ext^1_\ce(L,M)$ is one-dimensional
(and hence $\Ext^1_\ce(M,L)$ is one-dimensional) and
\[
0\to L \to E \to M \to 0 \mbox{ and } 0\to M \to E' \to L \to 0
\]
are non-split triangles, then we have
\begin{equation} \label{eq:ExactClusterCharacterExchangeEquation}
\zeta(L) \zeta(M) = \zeta(E) + \zeta(E').
\end{equation}
\end{itemize}
\end{definition}

From now on, we assume in addition that $\ce$ contains a cluster-tilting object $T$.
Using $T$ we will construct a cluster character on $\ce$ and link it
to Palu's cluster character associated with the image of $T$ in the
triangulated category $\cc=\ul{\ce}$ (\confer section~\ref{ss:ClusterCharacters}).

Let $C$ be the endomorphism algebra of $T$ (in $\ce$) and $\ul{C}=\End_{\cc}(T)$. Let
\[
F=\Hom_\ce(T,?) : \ce \to \mod C,
\]
\[
 G=\Hom_\cc(T,?) : \cc \to \mod \ul{C}.
\]
Let $T_i$, $1\leq i\leq n$, be the pairwise non isomorphic
indecomposable direct summands of $T$. We choose the numbering of
the $T_i$ so that $T_i$ is projective exactly for $r<i\leq n$ for
some integer $1\leq r\leq n$. For $1\leq i\leq n$, let $S_i$ be the
top of the indecomposable projective $P_i=FT_i$. Note that $C$ and
$\ul{C}$ are finite dimensional $k$-algebras, so finitely presented
modules are the same as finitely generated modules. As in section~4
of \cite{KellerReiten07}, we identify $\Mod \ul{C}$ with the full
subcategory of $\Mod C$ formed by the modules without composition
factors isomorphic to one of the $S_i$, $r< i\leq n$. Let $\cd C$ be
the unbounded derived category of the abelian category $\Mod C$ of
all right $C$-modules. Let $\per C$ be the perfect derived category
of $C$, \ie the full subcategory of $\cd C$ whose objects are all
the complexes quasi-isomorphic to bounded complexes of finitely
generated projective $C$-modules. Let $\cd^b(\mod C)$ the bounded
derived category of $\mod C$ identified with the full subcategory of
$\cd C$ whose objects are all complexes whose total homology is
finite-dimensional over $k$. As shown in section~4 of
\cite{KellerReiten07}, we have the following embeddings
\[
 \mod \ul{C} \hookrightarrow \per C \hookrightarrow \cd^b(\mod C).
\]
We have a bilinear form
\[
 \langle \ , \ \rangle : K_0(\per C) \times K_0(\cd^b(\mod C))\longrightarrow \Z
\]
defined by
\[
 \langle [P], [X]\rangle  = \sum (-1)^i \mbox{dim}\ \Hom_{\cd^b(\mod C)}(P, \Sigma^i X),
\]
where $K_0(\per C)$ (resp. $K_0(\cd^b(\mod C))$) is the Grothendieck group of $\per C$
(resp. $\cd^b(\mod C) $) and  $\Sigma$ is the shift functor of $\cd^b(\mod C)$.

For arbitrary finitely generated $C$-modules $L$ and $N$, put
\[
[L,N]= \, ^0[L,N]=\mbox{dim}_k\Hom_C(L,N)\  \  \mbox{and} \ \
 ^i[L,N]=\mbox{dim}_k\Ext_C^i(L,N) \ \mbox{for}\ \ i\geq 1.
\]
Let
\[
 \langle L, N\rangle_{\tau} =  [L,N]- \, ^1[L,N]\ \mbox{and} \ \langle L,N\rangle_3 = \sum_{i=0}^3\
(-1)^i \ ^i[L,N]
\]
be the truncated Euler forms on the split Grothendieck group $K_0^{sp}(\mod C)$.
By the proposition below, if $L$ is a $\ul{C}$-module, then $\langle L, N\rangle_3$
only depends on the dimension vector $\dimv \ L$ in $K_0(\mod C)$. We put
\[
 \langle \dimv \ L, N\rangle_3 = \langle L, N \rangle_3.
\]

\begin{proposition}
\label{prop1}
\begin{itemize}
\item [a)]  The restriction of the map
\[
 K_0(\per C) \longrightarrow K_0(D^b(\mod C))=K_0(\mod C)
\]
induced by the inclusion of $\per C$ into $\cd^b(\mod C)$
to the subgroup generated by the $[S_i]$, $1\leq i\leq r$, is injective.
\item [b)] If $L$, $N$ are two $\ul{C}$-modules such that $\dimv \ L = \dimv \ N$ in $K_0(\mod C)$, then
\[
\langle L, Y\rangle_3 = \langle N, Y\rangle_3
\]
for each finitely generated $C$-module $Y$.
\end{itemize}
\end{proposition}

\begin{proof} a) We need to show that for arbitrary finitely generated
$\ul{C}$-modules $L$, $N$
with $\ul{\mbox{dim}}\ L = \ul{\mbox{dim}}\ N$, we have $[L]=[N]$ in $K_0(\per C)$.
 Let
\[
 0=L_s\subset L_{s-1}\subset  \cdots \subset L_0=L
\]
and
\[
 0=N_s\subset N_{s-1}\subset  \cdots \subset N_0=N
\]
be composition series of $L$ and $N$ respectively. By
\cite{KellerReiten07},  we know that every $\ul{C}$-module has
projective dimension at most $3$ in $\mod C$. Assume for simplicity
that $L_{s-1}=S_1$, $L_{s-2}/L_{s-1}=S_2$. Denote by $P_i^{*}$ a
minimal  projective resolution of $S_i$. Then we have the following
commutative diagram
\[
\xymatrix{0 \ar[r]  & P_1^3 \ar[r] \ar[d] & P_1^2 \ar[r] \ar[d] & P_1^1 \ar[r]  \ar[d] & P_1^0 \ar[r] \ar[d] & L_{s-1} \ar[r] \ar[d] & 0\\
0\ar[r] & P_1^3\oplus P_2^3 \ar[r] \ar[d] & P_1^2\oplus P_2^2 \ar[r] \ar[d] & P_1^1\oplus P_2^1 \ar[r] \ar[d]& P_1^0\oplus  P_2^0 \ar[r] \ar[d] & L_{s-2} \ar[r] \ar[d] & 0\\
0\ar[r] & P_2^3 \ar[r] & P_2^2 \ar[r] & P_2^1 \ar[r] &  P_2^0 \ar[r] & L_{s-2}/L_{s-1} \ar[r] & 0}
\]
where the middle term is a projective resolution of $L_{s-2}$.
In this way, we inductively construct projective resolutions for $L$ and $N$.
If $m_i$ is the multiplicity of $S_i$ in the composition factors of $L$ and $N$,
then we obtain projective resolutions of $L$ and $N$ of the form
\[
 0\to \bigoplus_{i=1}^r(P_i^3)^{m_i}\xrightarrow{f_3}\bigoplus_{i=1}^r(P_i^2)^{m_i}
 \xrightarrow{f_2}\bigoplus_{i=1}^r(P_i^1)^{m_i}\xrightarrow{f_1}\bigoplus_{i=1}^r(P_i^0)^{m_i}\to L\to 0,
\]
\[
0\to \bigoplus_{i=1}^r(P_i^3)^{m_i}\xrightarrow{g_3}\bigoplus_{i=1}^r(P_i^2)^{m_i}
\xrightarrow{g_2}\bigoplus_{i=1}^r(P_i^1)^{m_i}\xrightarrow{g_1}\bigoplus_{i=1}^r(P_i^0)^{m_i}\to N\to 0.
\]
Let $P^L$ ({\em resp.} $P^N$) be the projective resolution complex of $L$ ({\em resp.} $N$).
We have $L\cong P^L$ and $N\cong P^N$ in $\per C$, which implies $[L]=[P^L]=[P^N]=[N]$
in $K_0(\per C)$.

b) We have
\[
\langle L, Y\rangle_3 = \langle P^L, Y \rangle = \langle [P^L], [Y]\rangle,
  \]
\[
 \langle N, Y\rangle_3 = \langle P^N, Y \rangle = \langle [P^N], [Y]\rangle.
\]
By a), we have $[P^L]=[P^N]$ in $K_0(\per C)$, which implies the equality.
\end{proof}

One should note that the truncated Euler form $\langle \ , \ \rangle_3$
does not descend to the Grothendieck group $K_0(\mod C)$ in general
(except if the global dimension of $C$ is not greater than 3),
\confer remark~\ref{remark:gldim3}.
\vspace{0.5cm}

Using the bilinear forms introduced so far,
for $M\in\ce$, we define the Laurent polynomial
\[
X'_M = \prod_{i=1}^n x_i^{\langle FM, S_i\rangle_{\tau}}
\sum_e \chi(Gr_e(\Ext^1_\ce(T,M))) \, \prod_{i=1}^n x_i^{-\langle e,S_i\rangle_3}.
\]
Here we consider $\Ext^1_\ce(T,M)$ as a right $C$-module via the natural
action of $C=\End_\ce(T)$ on the first argument; the sum ranges over all
the elements of the Grothendieck group; for a $C$-module $L$, the notation $Gr_e(L)$ denotes the
projective variety of submodules of $L$ whose class in the Grothendieck group is $e$; for
an algebraic variety $V$, the notation $\chi(V)$ denotes the Euler characteristic
(of the underlying topological space of $V$ if $k=\C$ and of $l$-adic
cohomology if $k$ is arbitrary).

Since $\cc$ is 2-Calabi-Yau, the object $\ul{T}=\oplus_{i=1}^r T_i$ is a
cluster tilting object of $\cc$. For an object $M$ of $\cc$, put
\[
X_M = X^{\ul{T}}_M \ko
\]
where $M \mapsto X^{\ul{T}}_M$ is Palu's cluster character associated
with the cluster-tilting object $\ul{T}$, \confer section~\ref{ss:ClusterCharacters}.

The following theorem shows that $M\mapsto X'_M$ is a cluster character
on $\ce$ and that, if we specialize the `coefficients' $x_{r+1}, \ldots, x_n$
to $1$, it specializes to the composition of Palu's cluster character
$M \mapsto X_M$ with the suspension functor $M \mapsto \Sigma M$.
Notice that this theorem does not involve cluster algebras (but paves
the way for establishing a link with cluster algebras when $\ce$
admits a cluster structure, \confer~Theorem~\ref{thm2} below).
\begin{theorem}\label{thm1}
As above, let $k$ be an algebraically closed field and $\ce$ a $k$-linear
Frobenius category with split idempotents such that
$\ce$ is $\Hom$-finite, the stable category $\cc=\ul{\ce}$ is $2$-Calabi-Yau
and $\ce$ contains a cluster-tilting object $T$. For an
object $M$ of $\ce$, let $X'_M$ and $X_M$ be the Laurent
polynomials defined above.
\begin{itemize}
\item[a)] We have $X'_{T_i}=x_i$ for $1\leq i \leq n$.
\item[b)] The specialization of $X'_M$ at $x_{r+1}=x_{r+2}= \ldots = 1$
is $X_{\Sigma M}$, where $\Sigma$ is the suspension of $\cc$.
\item[c)] For any two objects $L$ and $M$ of $\ce$, we have
$X'_{L\oplus M} = X'_L X'_M$.
\item[d)] If $L$ and $M$ are objects of $\ce$ such that $\Ext_\ce^1(L,M)$
is one-dimensional and we have non split conflations
\[
0 \to L \to E \to M \to 0 \mbox{ and } 0 \to M \to E' \to L \to 0  \ko
\]
then we have
\[
X'_L X'_M = X'_E + X'_{E'}.
\]
\end{itemize}
\end{theorem}
\begin{proof}
a) is straightforward.

b)  We have
\[
 X_{\Sigma M}=\prod_{i=1}^r x_i^{-[\operatorname{coind}_{\ul{T}} (\Sigma M): T_i]}
\sum_e \chi(Gr_e(G\Sigma M)) \, \prod_{i=1}^r x_i^{\langle S_i,e\rangle_a}.
\]
Now by the definition, we have
\[
 G\Sigma M = \Hom_{\cc}(T,\Sigma M) = \Ext_{\ce}(T,M).
\]
Therefore, we only need to show that the exponents of $x_i$, $1\leq i\leq r$, in the corresponding terms of $X_{\Sigma M}$ and $X_M'$ are equal.
There exists a triangle in $\cc$
\[
 T_M^1\to T_M^0 \to M\to \Sigma T_1
\]
with $T_M^0$ and $T_M^1$ in $\add{\ul{T}}$. We may and will assume that this triangle is minimal, {\em i.e.} does not admit a non zero direct factor of the form
\[
 T'\to T'\to 0 \to \Sigma T'.
\]
Since $\ce$ is Frobenius,  we can lift this triangle to a short exact sequence in $\ce$
\[
 0\to T_M^1\to T_M^0\oplus  P\to M\to 0,
\]
where $P$ is a projective of $\ce$. Applying the functor $F$ to this short exact sequence, we get a
projective resolution of  $FM$ as a $C$-module,
\[
 0\to FT_M^1\to F(T_M^0\oplus P) \to FM \to 0.
\]
Therefore, we have
\[
  \langle FM,S_i\rangle_{\tau} = [FT_M^0\oplus FP, S_i]-[FT_M^1, S_i] = [FT_M^0,S_i]-[FT_M^1,S_i]
 \]
for  $1\leq i\leq r$.

On the other side, we have the following minimal triangle
\[
 \Sigma M\to \Sigma^2 T_M^1\to \Sigma^2 T_M^0\to \Sigma^2 M.
\]
By the definition of the coindex, we get
\[
 - [\mbox{coind}_{\ul{T}}\ (\Sigma M):T_i] = - [T_M^1 - T_M^0 :T_i] = \langle FM,S_i\rangle_{\tau},\ \mbox{for}\ \, 1\leq i\leq r.
\]

Next we will show that $\langle S_i, e\rangle_{a}=-\langle e,
S_i\rangle_3.$ Let $N$ be a $C$-module such that $\mbox{\ul{dim}}\
N=e$. Note that $N$ and the $S_i, \ 1\leq i\leq r$, are $\ul{C}$-modules
and that all of them are finitely presented $C$-modules. Therefore, they
lie in the perfect derived category $\per (C)$. Thus, we can use the
relative 3-Calabi-Yau property of $\per (C)$ ({\em \confer
\cite{KellerReiten07}}) to deduce that $\langle S_i,
e\rangle_a=-\langle e, S_i\rangle_3$. We have
\[
\Ext_C^2(N,S_i) =\Ext_C(S_i,N)=\Ext_{\ul{C}}(S_i,N),
\]
\[
 \Ext_C^3(N,S_i) =\Hom_C(S_i,N)=\Hom_{\ul{C}}(S_i,N),
\]
for $1\leq i\leq r$.
By the definition of $\langle S_i, N\rangle_{a}$, we have
\begin{eqnarray*}
 \langle S_i, N\rangle_{a} &=&\mbox{dim}_k\Hom_{\ul{C}}(S_i,N)-\mbox{dim}_k\Ext_{\ul{C}}(S_i,N)+
 \mbox{dim}_k\Ext_{\ul{C}}(N,S_i)-\mbox{dim}_k\Hom_{\ul{C}}(N,S_i)\\
&=&\mbox{dim}_k\Hom_{C}(S_i,N)-\mbox{dim}_k\Ext_{C}(S_i,N)+\mbox{dim}_k\Ext_C(N,S_i)-\mbox{dim}_k\Hom_C(N,S_i)\\
&=&^3[N,S_i]-\,^2[N,S_i]+\,^1[N,S_i]-[N,S_i]\\
&=&-\langle N, S_i\rangle_3.
\end{eqnarray*}

c) is proved in exactly the same way as Corollary~3.7 in \cite{CalderoChapoton06}.

d)  Let
\[
0 \to L \xrightarrow{i} E \xrightarrow{p} M \to 0 \ \,\mbox{ and }\ \, 0 \to M \xrightarrow{i'} E' \xrightarrow{p'} L \to 0  \ko
\]
be the non-split conflations in $\ce$, and
\[\Sigma L\xrightarrow{G\Sigma i} \Sigma E\xrightarrow{G\Sigma p} \Sigma M \to \Sigma^2 L
\]
\[
 \Sigma M\xrightarrow{G\Sigma i'} \Sigma E'\xrightarrow{G\Sigma p'} \Sigma L \to \Sigma^2 N
\]
 the associated triangles in $\cc$. For any classes $e$, $f$, $g$ in the Grothendieck group
 $K_0(\mod \ul{C})$, let $X_{e,f}$ be the variety whose points are the $\ul{C}$-submodules
 $E\subset G\Sigma E$ such that the dimension vector of $(G\Sigma i)^{-1}E$
 equals $e$ and the dimension vector of $(G\Sigma p)E$ equals $f$. Similarly,
 let $Y_{f,e}$ be the variety whose points are the $\ul{C}$-submodules
 $E\subset G\Sigma E'$ such that the dimension vector of $(G\Sigma i')^{-1}E$
 equals $f$ and the dimension vector of $(G\Sigma p')E$ equals $e$. Put
\[
 X_{e,f}^g=X_{e,f}\cap Gr_g(G\Sigma E),
\]
\[
 Y_{f,e}^g=Y_{f,e}\cap Gr_g(G\Sigma E').
\]
Since $\cc$ is a $2$-CY triangulated category, by section 5.1 of
\cite{Palu08} we also have
\[
 \chi (Gr_e(G\Sigma L)\times Gr_f(G\Sigma M))=\sum_g\chi(X_{e,f}^g)+\chi(Y_{f,e}^g).
\]
Therefore, part d) is a consequence of the following lemma.
\begin{lemma}
\label{lem1}
If $X_{e,f}^g\neq \emptyset$, then we have the following equality
\[
 -\langle g, S_i\rangle_3 +\langle FE, S_i\rangle_{\tau}  = -\langle e+f , S_i\rangle_3 +\langle FL, S_i\rangle_{\tau} +\langle FM, S_i\rangle_{\tau}, 1\leq i\leq n.
\]

\end{lemma}
{\it Proof.} We have the following commutative diagram as in section
4 of \cite{Palu08}
\[
\xymatrix{(G\Sigma i)^{-1}E \ar[r]^-{\alpha} \ar[d]_i & E \ar[r]^-{\beta} \ar[d]_j & (G\Sigma p)E \ar[r] \ar[d]_k & 0\\
G\Sigma L \ar[r]^-{G\Sigma i} & G\Sigma E \ar[r]^-{G\Sigma p}  &
G\Sigma M \ar[r] & G\Sigma^2 L}
\]
where $i$, $j$, $k$ are monomorphisms, $\beta$ is an epimorphism and $[E]=g$,
$[G\Sigma i)^{-1}E]=e$, $[G\Sigma p)E]=f$ in $K_0(\mod C)$.
One can easily show that $\ker G\Sigma i = \ker \alpha$.  We have an exact sequence
\[
 0\to \ker \alpha \to (G\Sigma i)^{-1}E \to E \to (G\Sigma p)E\to 0.
\]
 If we apply $F=\Hom_{\ce}(T,?)$ to the short exact sequence
\[
 0\to L\to E\to M\to 0,
\]
we get the long exact sequences of $C$-modules
\[
 0\to FL\to FE\to FM\to G\Sigma L\xrightarrow{G\Sigma i}\ G \Sigma E\to \ldots,
\]
and
\[
 0\to FL\xrightarrow{Fi} FE\xrightarrow{Fp} FM\to \ker \alpha\to 0.
\]
Since $\ker \alpha$, $(G\Sigma i)^{-1}E$, $E$, $(G\Sigma p)E$ are $\ul{C}$-modules,
and the projective dimensions of $FL$, $FE$, $FM$ are not greater than 1, we can
use the method of Proposition ~\ref{prop1} to construct the projective resolutions
and compute the  truncated Euler forms.  We get that
\[
 \langle e,S_i\rangle_3 +\langle f, S_i\rangle_3 = \langle g,S_i\rangle_3 +\langle \ker \alpha, S_i\rangle_3,
\]
and
\[
 \langle FL,S_i\rangle_3+\langle FM,S_i\rangle_3=\langle FE,S_i\rangle_3+\langle \ker \alpha,S_i\rangle_3.
\]
Note that $\langle FL,S_i\rangle_3=\langle FL,S_i\rangle_{\tau}$, $\langle FM,S_i\rangle_3=\langle FM,S_i\rangle_{\tau}$ and $\langle FE,S_i\rangle_3=\langle FE,S_i\rangle_{\tau}$, which implies
\[
 \langle FL,S_i\rangle_{\tau}+\langle FM,S_i\rangle_{\tau}-
 \langle e+f,S_i\rangle_3=\langle FE,S_i\rangle_{\tau}-\langle g,S_i\rangle_3.
\]
\end{proof}

\begin{remark} \label{remark:gldim3}
If $C$ has finite global dimension, the Grothendieck group $K_0(\mod C)$ has
the Euler form $\langle \ , \ \rangle$. We can then define a Laurent polynomial
$X_M^f$ as follows
\[
 X_M^f=\prod_{i=1}^n x_i^{\langle FM, S_i\rangle}
\sum_e \chi(Gr_e(\Ext^1_\ce(T,M))) \, \prod_{i=1}^n x_i^{\langle S_i,e\rangle}.
\]
One can show that in this case $X_M'=X_M^f$.
In fact, if $\mbox{gldim} C <\infty$, then the perfect derived category $\per (C)$
equals $D^b(\mod C)$, and $S_i$ belongs to $\per (C)$ for all $i$.
Thus, we have
\[
\langle S_i,e\rangle =\sum_{i=0}^3(-1)^i \mbox{dim}\Ext^i_C(S_i,e)=-\langle e,S_i\rangle_3
\]
and $\langle FM, S_i\rangle_{\tau}=\langle FM, S_i\rangle$. The assumption that $C$ is of finite
global dimension holds for the examples constructed in \cite{BuanIyamaReitenScott07} by Proposition~I.2.5~b)
of [loc. cit.] and for the examples constructed in \cite{GeissLeclercSchroeer06} by
Proposition~11.5 of [loc.cit.].
\end{remark}

\section{Index and $\mathbf{g}$-vector}
\label{IG}

\subsection{Index} \label{ss:Index}
As in section~\ref{CC2FC}, we let $k$ be an algebraically closed
field and $\ce$ a $k$-linear Frobenius category with split
idempotents. We assume that $\ce$ is $\Hom$-finite and that the
stable category $\cc=\ul{\ce}$ is $2$-Calabi-Yau (\confer
section~\ref{ss:CalabiYauCategories}). Moreover, we assume that
$\ce$ admits a cluster-tilting object $T$ and we write
$C=\End_\ce(T)$ and $\ul{C}=\End_\cc(T)$.

Let $\cd(\Mod C)$ be the derived category of $C$-modules, $\cd^-(\mod C)$
the  right bounded derived category of $\mod C$, $\ch^-(\cp)$ the right
bounded homotopy category of finitely generated projective $C$-modules.
It is well known that there is an equivalence
\[
 \ch^-(\cp)\stackrel{\sim}{\longrightarrow}\cd^-(\mod C).
\]

\begin{proposition}\label{prop2}
 For an arbitrary $\ul{C}$-module $Z$ which is also a finitely presented $C$-module  we have a canonical isomorphism
\[
 D \Hom_{\cd^-(\mod C)}(Z,?) \stackrel{\sim}{\longrightarrow} \Hom_{\cd^-(\mod C)}(?,Z[3]).
\]
\end{proposition}

\begin{proof} For arbitrary $X\in \cd^-(\mod C)$, by the equivalence,
we have a $P_X\in \ch^-(\cp)$  such that $X\cong P_X$ in $\cd^-(\mod C)$.
Assume that $P_X$ has the following form
\[
 \ldots \to P_m\to P_{m+1}\to \ldots \to P_{n-1}\to P_n\to 0\to 0\ldots.
\]
Put
\begin{eqnarray*}
X_0 &=&   \ldots \to 0 \to 0\to P_n\to 0\ldots,\\
X_i &=&\ldots \to 0\to P_{n-i} \to \ldots \to P_n \to 0 \ldots, \
\mbox{for}\  i>0.
\end{eqnarray*}
Clearly, the complex $P_X$ is the direct limit of the complexes
$X_i$. We write $\mbox{hocolim}$ for the total left derived functor
of the functor of taking the direct limit. Since taking direct
limits over filtered systems is an exact functor, the functor
$\mbox{hocolim}$ is simply induced by the direct limit functor.
Thus, we have $P_X \cong \mbox{hocolim} \ X_i$ in $\cd(\Mod C)$.
Note that by Proposition 4 of \cite{KellerReiten07}, $Z$ belongs to
$\per C$, \ie  $Z$ is compact in $\cd(\Mod C)$. So we have
\begin{eqnarray*}
\Hom_{\cd(\Mod C)}(Z,X) &\cong& \Hom_{\cd(\Mod C)}(Z, P_X)\\
&\cong& \Hom_{\cd(\Mod C)}(Z, \mbox{hocolim} X_i)\\
 &\cong& \mbox{colim} \Hom_{\cd(\Mod C)}(Z,X_i).
\end{eqnarray*}

By the definition of $X_i$, we know that $X_i \in \per C$.
Since $\per C$ is a full subcategory of $\cd(\Mod C)$,
by the relative 3-Calabi-Yau property of $\per C$, we have the following
\[
 \mbox{colim} \Hom_{\cd(\Mod C)}(Z,X_i) \cong \mbox{colim} D\Hom_{\cd(\Mod C)}(X_i, Z[3]).
\]
It is easy to see that this colimit is a stationary system,
{\em i.e.} $\exists\  N$ such that for  $i>N$, we have
\[
 D\Hom_{\cd(\Mod C)}(X_i, Z[3])\cong D\Hom_{\cd(\Mod C)}(X_{i+1}, Z[3]).
\]
Thus, we have
\begin{eqnarray*}
\mbox{colim} D\Hom_{\cd(\Mod C)}(X_i, Z[3])&\cong & D \ \mbox{lim} \Hom_{\cd(\Mod C)}(X_i, Z[3])\\
&\cong & D\Hom_{\cd(\Mod C)}(\mbox{hocolim} X_i, Z[3])\\
&\cong & D\Hom_{\cd(\Mod C)}(P_X, Z[3]).
\end{eqnarray*}
Note that since $\cd^-(\mod C)$ is a full subcategory of $\cd(\Mod C)$, we get the isomorphism
\[
 D \Hom_{\cd^-(\mod C)}(Z,X) \stackrel{\sim}{\longrightarrow} \Hom_{\cd^-(\mod C)}(X,Z[3]).
\]
\end{proof}

For each $X\in \ce$, there is a unique  minimal conflation (up to isomorphism)
\[
 0\to T_X^1\to T_X^0\to X\to 0
\]
with $T_X^0,T_X^1\in \add T$. As in \cite{Palu08}, put
\[
 \mbox{ind}_T(X)=[T_X^0]-[T_X^1] \ \ \mbox{in} \ K_0(\add T).
\]
By the proof of Theorem~\ref{thm1}, we have
\[
 \mbox{ind}_T(X) = \sum_{i=1}^n \langle FX, S_i\rangle_{\tau} [T_i].
\]

The following result is easily deduced from Theorem~2.3 of \cite{DehyKeller08}.
\begin{lemma}
\label{lem2}
  If $X$ is a rigid object of $\ce$, then $X$ is determined up to isomorphism by
  $\mbox{ind}_T(X)$,  i.e. if \ $Y$ is rigid and $\mbox{ind}_T(X) = \mbox{ind}_T(Y)$,
  then $X$ is isomorphic to $Y$.
\end{lemma}
\begin{proof}
 Since $\mbox{ind}_T(X)=\mbox{ind}_T(Y)$,  we have $\mbox{ind}_{\ul{T}}(X)=\mbox{ind}_{\ul{T}}(Y)$ in the stable
 category $\ul{\ce}$. By Theorem~2.3 of \cite{DehyKeller08}, we have $X\cong Y$ in $\ul{\ce}$.
 Thus, there are $\ce$-projectives $P_X$ and $P_Y$ such that $X\oplus P_X\cong Y\oplus P_Y$ in $\ce$.
 Consider the minimal right $T$-approximation of $X\oplus P_X$
\[
 0\to T^1\to T^0\to X\oplus P_X\to 0,
\]
we have $\mbox{ind}_T(X\oplus P_X) =\mbox{ind}_T(Y\oplus P_Y)=[T^0]-[T^1]$. Note that
\[
 \mbox{ind}_T(X)=\mbox{ind}_T(X\oplus P_X)-[P_X]=\mbox{ind}_T(Y\oplus P_Y)-[P_Y]=\mbox{ind}_T(Y),
\]
which implies $[P_X]=[P_Y]$ in $K_0(\add T)$. Thus, we have $P_X\cong P_Y$ and $X\cong Y$ in $\ce$.
\end{proof}

\subsection{$\mathbf{g}$-vector} \label{ss:gvector}
Let us recall the definition of $\mathbf{g}$-vectors from section~7
of \cite{FominZelevinsky07}. Let $1< r\leq n$ be integers. Let
$\tilde{B}=(\tilde{b}_{ij})$ be an $n\times r$ matrix with integer
entries, whose principal part $B$ (\ie the submatrix formed by the
first $r$ rows) is antisymmetric. Let $\ca(\tilde{B})$ be the
cluster algebra with coefficients associated with $\tilde{B}$, \confer
the end of section~\ref{ss:ClusterAlgebras}. Let $z$ be an element
of $\ca(\tilde{B})$. Suppose that we can write $z$ as
\[
 z = R(\hat{y}_1,\ldots, \hat{y}_r)\prod_{i=1}^nx_i^{g_i},
 \mbox{where} \   \hat{y}_j=\prod_{i=1}^nx_i^{\tilde{b}_{ij}},
\]
where  $R(\hat{y}_1,\ldots, \hat{y}_r)$ is a primitive rational
polynomial. If $\mbox{rank}\, \tilde{B}=r$, then the
$\mathbf{g}$-vector of $z$ is defined by
\[
 g(z)=(g_1,\ldots, g_r).
\]
Note that $\mbox{rank}\,\tilde{B}=r$  implies that the $\mathbf{g}$-vector is well-defined.

As in the previous section, we let $k$ be an algebraically closed
field and $\ce$ a $k$-linear Frobenius category with split
idempotents. We assume that $\ce$ is $\Hom$-finite and that the
stable category $\cc=\ul{\ce}$ is $2$-Calabi-Yau (\confer
section~\ref{ss:CalabiYauCategories}). Moreover, we assume that
$\ce$ admits a cluster-tilting object $T$ and we write
$C=\End_\ce(T)$ and $\ul{C}=\End_\cc(T)$. Let $T_1$, $T_2$,
$\ldots$, $T_n$ be the pairwise non isomorphic indecomposable direct
summands of $T$ numbered in such a way that $T_i$ is projective iff
$r< i\leq n$. We define $B(T)=(b_{ij})_{n\times n}$ to be the
antisymmetric matrix associated with the  quiver of the endomorphism
algebra of $T$. Let $B(T)^0$ be the submatrix formed by the first
$r$ columns of $B(T)$. We suppose that we have $\rank\, B(T)^0=r$.
In analogy with the definition of $\mathbf{g}$-vectors in a cluster
algebra, for $M\in \ce$, if we can write $X_M'$ as
\[
 X_M'=R(\hat{y}_1,\ldots, \hat{y}_r)\prod_{i=1}^nx_i^{g_i},
 \mbox{where} \   \hat{y}_j=\prod_{i=1}^nx_i^{b_{ij}},
\]
where $R(\hat{y}_1,\ldots, \hat{y}_r)$ is a primitive rational
polynomial,  then we define the $\mathbf{g}$-vector $g_T(X_M')$ of
$M$ with respect to $T$ to be
\[
 g_T(X_M')=(g_1,\ldots, g_r).
\]
As in the cluster algebra case, this is well-defined since $\rank\, B(T)^0=r$.

\begin{proposition}\label{prop3}
Assume that $\rank\,  B(T)^0 = r$.
For arbitrary $M\in \ce$, the $\mathbf{g}$-vector $g_T(X_M')$ is well-defined and its $i$-th coordinate  is given by
\[
 g_T(X_M')(i)=[\mbox{ind}_T(M):T_i],\  1\leq i\leq r.
\]
\end{proposition}
\begin{proof}    By the relative 3-Calabi-Yau property of $\cd^-(\mod C)$, for $1\leq i\leq n$, $1\leq j\leq r$,  we have
\begin{eqnarray*}
 \langle S_i, S_j\rangle_3 &=&[S_i, S_j]-\,^1\,\![S_i,S_j]+\,^2\,\![S_i,S_j]-\,^3\,\![S_i,S_j]\\
&=&[S_i, S_j]-\,^1\,\![S_i,S_j]+\,^1\,\![S_j,S_i]-\,[S_j,S_i]\\
&=&^1\,\![S_j,S_i]-\,^1\,\![S_i,S_j]\\
&=&b_{ij},
\end{eqnarray*}
where the last equality follows from the definition of $B(T)$.
 Recall the definition of $X_M'$
 \[
 X_M'=\prod_{i=1}^n x_i^{\langle FM, S_i\rangle_{\tau}}
\sum_e \chi(Gr_e(\Ext^1_\ce(T,M))) \ \prod_{i=1}^n x_i^{- \langle e, S_i\rangle_3}.
\]
Let $e$ be the dimension vector of a $C$-submodule of
$\Ext_{\ce}^1(T,M)$ and $e_j$ its $j$-th coordinate in the basis of
the $S_i$, $1\leq i\leq n$. Then we have
\[
 - \langle e, S_i\rangle_3 = - \sum_{j=1}^r e_j\langle S_j,S_i\rangle_3 =  \sum_{j=1}^rb_{ij}e_j.
\]
Therefore, we get
\[
\prod_{i=1}^n x_i^{- \langle e,S_i\rangle_3}=\prod_{i=1}^nx_i^{\sum_{j=1}^r b_{ij}e_j}=\prod_{j=1}^r\hat{y}_j\,^{e_j}.
\]
Thus, we can write
\[
 X_M'=\prod_{i=1}^n x_i^{\langle FM, S_i\rangle_{\tau}}
(\sum_e \chi(Gr_e(\Ext^1_\ce(T,M))) \  \prod_{j=1}^r \hat{y}_j\,^{e_j}).
\]
The polynomial
\[
 R(\hat{y}_1, \ldots, \hat{y}_r) = \sum_e \chi(Gr_e(\Ext_{\ce}^1(T,M))) \prod_{j=1}^r \hat{y}_j\,^{e_j}
\]
is primitive since it has constant term $1$. Thus, by definition we have  $g_T(X_M')(i)=\langle FM, S_i\rangle_{\tau} =  [\mbox{ind}_T(M): T_i]$.
\end{proof}

\begin{corollary}\label{cor1} As above, let $\ce$ be a $\Hom$-finite
$k$-linear Frobenius category such that its stable category
$\cc=\ul{\ce}$ is $2$-Calabi-Yau and assume that
\begin{itemize}
\item $\ce$ admits a cluster-tilting object $T$ with indecomposable direct summands
$T_1$, \ldots, $T_n$ numbered in such a way that $T_i$ is projective
iff $r<i\leq n$, where $1<r\leq n$ is an integer;
\item the first $r$ columns of the antisymmetric matrix $B(T)$ associated with
the quiver of the algebra $C=\End_\ce(T)$ are linearly independent.
\end{itemize}
Then the following hold.
\begin{itemize}
\item[a)] The map $M \mapsto X'_M$ induces an injection from the set
of isomorphism classes of non projective rigid indecomposables of
$\ce$ into the set $\Q(x_1, \ldots, x_n)$.
\item[b)] Let $I$ be a
finite set and $T^i$, $i\in I$, cluster tilting objects of $\ce$.
Suppose that for each $i\in I$, we are given an object $M_i$ which
belongs to $\add T^i$ and does not have non zero projective direct
factors. If the $M_i$ are pairwise non isomorphic, then the
$X_{M_i}'$ are linearly independent.
\end{itemize}
\end{corollary}
\begin{proof}  a) clearly follows from b). Let us prove b).
First, we will show that we can assign a degree to each $x_i$ such
that for every $1 \leq i \leq r$ the degree of $\hat{y}_i$ is 1.

Indeed, it suffices to put $\mbox{deg}(x_i)=k_i$, where the $k_i$ are rationals such that we have
\[
 (k_1, k_2, \ldots, k_n)B(T)^0=(1,1,\ldots, 1).
\]
Since $\rank\, B(T)^0 =r$, this equation  does admit a solution.
Thus, the term of strictly minimal total degree in  $X_{M_j}'$ is
\[
 \prod_{i=1}^nx_i^{[\operatorname{ind}_T(M_j):T_i]}.
\]

Suppose that the $X_{M_i}'$ are linearly dependent, {\em i.e.} there
is a non-empty subset $I'$ of $I$ and rationals $c_i$, $i\in I'$,
which are all non zero such that
\[
 \sum_{i\in I'} c_i X_{M_i}'=0.
\]

If we consider the terms of minimal total degree of the polynomial above, we find
\[
 \sum_{j\in I''} c_j \prod_{i=1}^n x_i^{[\operatorname{ind}_T(M_j):T_i]}=0
\]
for some non-empty subset $I''$ of $I$. Since the  $M_j$ are all
pairwise non isomorphic, Lemma~\ref{lem2} implies that the indices
$\mbox{ind}_T(M_j)$ are all distinct.  Thus,  the monomials
$\prod_{i=1}^n x_i^{[\operatorname{ind}_T(M_j):T_i]}$ are linearly
independent. Contradiction.
\end{proof}

\begin{remark} If the algebra $C$ has finite global dimension, then
the condition $\rank B(T)^0=r$ is superfluous. Indeed, let $A$ be
the Cartan matrix of $C$. Then $B(T)^0$ is the submatrix formed by
the first $r$ columns of the invertible matrix $A^{-t}$.
\end{remark}

\vspace{0.2cm}

Next we will investigate the relation  between the indices of an exchange pair.

Recall that $F$ is the functor $\Hom_{\ce}(T,?): \ce \to \mod C$. A
conflation of $\ce$
\[
 0\to X\to Y\to Z\to 0
\]
is $F$-{\em exact} if
\[
 0\to FX\to FY\to FZ\to 0
\]
is exact in $\mod C$. The $F$-exact sequences define a new exact
structure on the additive category $\ce$. For each $X$, we have an
$F$-exact conflation
\[
 0\to T_1\to T_0\to X\to 0.
\]
This shows that $\ce$ endowed with the $F$-exact sequences has
enough projectives and that its subcategory of projectives is $\add
T$. Moreover, if we denote by $\Ext_F^i(X,Z)$ the $i$-th extension
groups of the category $\ce$ endowed with the $F$-exact sequences,
then $\Ext_F^1(X,Z)$ is the cohomology at $\Hom_{\ce}(T_1,Z)$ of the
complex
\[
 0\to \Hom_{\ce}(X,Z)\to \Hom_{\ce}(T_0, Z)\to \Hom_{\ce}(T_1,Z)\to 0\to \ldots.
\]

\begin{lemma}\label{lem3}
 For $X, Z\in \ce$, there is a functorial isomorphism
\[
 \Ext_F^i(X,Z) \stackrel{\sim}{\longrightarrow} \Ext_C^i(FX, FZ).
\]
\end{lemma}
\begin{proof}
 Clearly, the derived functor
\[
 \mathbb{L} F: \cd^b(\ce)\to \cd^b(\mod C)
\]
 is fully faithful.
Thus, $\Ext_F^i(X,Z) \stackrel{\sim}{\longrightarrow} \Ext_C^i(FX,
FZ)$.
\end{proof}

Now  Proposition~15.4 of \cite{GeissLeclercSchroeer07b}  still holds
in our general setting.
\begin{proposition}\label{prop4}
Let $T$ and $R$ be cluster tilting objects of $\ce$. Let
\[
 \eta':0\to R_k\to R'\to R_k^*\to 0, \  \ \eta'':0\to R_k^*\to R''\to R_k\to 0
\]
be the two exchange sequences associated to an indecomposable direct
summand $R_k$ of $R$ which is not $\ce$-projective. Then exactly one
of $\eta'$ and $\eta''$ is  $F$-exact. Moreover, we have
\[
 \dimv  \Hom_{\ce}(T,R_k)+\dimv \Hom_{\ce}(T,R_k^*) =
 \max \{\dimv \Hom_{\ce}(T,R'),\dimv \Hom_{\ce}(T,R'')\}.
\]
\end{proposition}
\begin{proof}
 Using Lemma~\ref{lem3}, the proof is the same as
 that of proposition 15.4 in \cite{GeissLeclercSchroeer07b}.
\end{proof}

\begin{corollary}
 Under the assumptions of the above proposition,  put
\[
I'=\mbox{ind}_T(R')-\mbox{ind}_T(R_k),
\]
 \[
I''=\mbox{ind}_T(R'')-\mbox{ind}_T(R_k).
 \]
Then we have
 \begin{eqnarray*}
\mbox{ind}_T(R_k^*)=\left\{\begin{array}{ll} I',\, \mbox{if}\ \dimv FI'\geq \dimv FI'', \\
I'', \,\mbox{if}\ \dimv FI'\leq \dimv FI''.
\end{array}
\right.
\end{eqnarray*}
and exactly one of these cases occurs.
 Let $h(i) = [\mbox{ind}_T(R')-\mbox{ind}_T(R''):T_i]$, for $1\leq i \leq n$. Then  $h$ is a linear combination of the columns of $B(T)^0$.
\end{corollary}

\begin{proof}  The first part follows from Proposition ~\ref{prop4} directly, because the index is additive on $F$-exact sequences.

Since $(R_k, R_k^*)$ is an exchange pair, we have
\[
 X_{R_k}'X_{R_k^*}' =  X_{R'}'+X_{R''}'.
\]
For simplicity, we write
\[
 H_M = \sum_e\chi(Gr_e(\Ext_{\ce}(T, M))) \prod_{i=1}^r\hat{y}_i\, ^{e_i}
\]
for $X_M'$.
By Proposition~\ref{prop3},  we have
\begin{eqnarray*}
 \prod_{i=1}^nx_i^{[\operatorname{ind}_T(R_k)+\operatorname{ind}_T(R_k^*) : T_i]} H_{R_k}H_{R_k^*}
= \prod_{i=1}^nx_i^{[\operatorname{ind}_T(R'): T_i]}H_{R'} +\prod_{i=1}^nx_i^{[\operatorname{ind}_T(R''): T_i]}H_{R''}.
\end{eqnarray*}
Assume that $\mbox{ind}_T(R_k^*) = \mbox{ind}_T(R') - \mbox{ind}_T(R_k)$. We have
\[
 H_{R_k}H_{R_k^*} -H_{R'} = \prod_{i=1}^nx_i^{[\operatorname{ind}_T(R'')-\operatorname{ind}_T(R'): T_i]}H_{R''}.
\]
By comparing the  minimal total degree we get that
$\prod_{i=1}^nx_i^{[\operatorname{ind}_T(R'')-\operatorname{ind}_T(R'):
T_i]}$ is a monomial in $\hat{y}_i$, $1\leq i\leq r$, which implies
the result.
\end{proof}

\section{Frobenius 2-Calabi-Yau realizations}
\label{F2R}

Recall the bijection defined in section~\ref{ss:IceQuivers} between
antisymmetric integer $n\times n$-matrices and finite quivers
without loops nor 2-cycles with vertex set $\{1,2, \ldots, n\}$: The
quiver $Q$ corresponds to the matrix $B$ iff $b_{ij}>0$ exactly when
there are arrows from $i$ to $j$ in $Q$ and in this case their
number is $b_{ij}$.

We call an $n\times n$  antisymmetric integer matrix $B$ {\em
acyclic} if  the corresponding quiver $Q$ does not have oriented
cycles. Two matrices $B$ and $B'$ are called {\em mutation
equivalent} if we can obtain $B'$ from $B$ by  a series of matrix
mutations followed by conjugation with a permutation matrix.

Let $0\leq r<n$ be positive integers and let $(Q,F)$ be an ice
quiver (\confer~section~\ref{ss:IceQuivers})  with vertex set
$Q_0=\{1,\ldots, n\}$ and set of frozen vertices $F=\{r+1, \ldots,
n\}$. We define $\tilde{B}$ to be the $n\times r$ matrix formed by
the first $r$ columns of the skew-symmetric matrix associated with
$Q$ and we let $\ca(Q,F)=\ca(\tilde{B})$ be the cluster algebra with
coefficients associated with $\tilde{B}$, \confer
sections~\ref{ss:ClusterAlgebras} and \ref{ss:IceQuivers}.

\begin{definition} \label{def:Frobenius2CYRealization}
A {\em Frobenius 2-Calabi-Yau realization} of the cluster
algebra $\ca(\tilde{B})$ is a Frobenius category $\ce$ with a
cluster tilting object $T$ as in section~\ref{CC2FC} such that
\begin{itemize}
 \item[1)] $\ce$ has a cluster structure in the sense of
 \cite{BuanIyamaReitenScott07}, \confer
 section~\ref{ss:FrobeniusCategories};
 \item[2)] $T$ has exactly $n$ indecomposable pairwise non isomorphic summands
 $T_1$, $T_2$, $\ldots$, $T_n$ and among these, precisely $T_{r+1}$, $\ldots$,
 $T_n$ are projectives;
\item[3)] The matrix $\tilde{B}$ equals the matrix formed by the first
$r$ columns of the antisymmetric matrix associated with the quiver
of the endomorphism algebra of $T$ in $\ce$.
\end{itemize}
\end{definition}

\begin{remark} \label{remark:No2CYFrobeniusRealization}
Suppose we have a Frobenius $2$-CY realization of a cluster algebra
$\ca(Q,F)$ as above. Let $1\leq s\leq r$. Then by
Lemma~\ref{lemma:ExactMultiplicities} b), we have conflations
\[
 0\to T_s^*\to E\to T_s\to 0,
\]
\[
 0\to T_s\to E'\to T_s^*\to 0.
\]
Here the middle terms are the sums
\[
 E = \bigoplus_{b_{is}>0} T_i^{b_{is}} \ko E' = \bigoplus_{b_{is}<0} T_i^{-b_{is}}.
\]
Therefore, none of the first $r$ vertices of $Q$ can be a source or a sink.\\
\end{remark}

\begin{example} \label{example:Frobenius2CYRealization}
All quivers obtained from Theorem~2.3 of \cite{GeissLeclercSchroeer07b}
and, more generally, from Theorem~II.4.1 of \cite{BuanIyamaReitenScott07} admit Frobenius
$2$-Calabi-Yau realizations.
We illustrate this on the following specific case taken
from section~II.4 of [loc.~cit.]. Let $\Delta$ be the graph
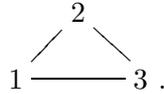
\[
\xymatrix@=0.4cm{ &2\ar@{-}[dr]\\ 1\ar@{-}[ur]\ar@{-}[rr] &&3 \ .}
\]
Let $\Lambda$ be the completion of the preprojective algebra of
$\Delta$ and $W$ the Weyl group associated with $\Delta$. Let
$w$ be the element of $W$ given by the reduced word
$s_2s_1s_2s_3s_2$. Let $e_i$,
$i=1, 2, 3$, be the primitive idempotents corresponding to the
vertices of $\Delta$. Let $I_i=\Lambda(1-e_i)\Lambda$.
By Theorem~II.2.8 of \cite{BuanIyamaReitenScott07}, the category $\mbox{Sub}\,\Lambda/I_w$
formed by all $\Lambda$-submodules of finite direct sums of
copies of $\Lambda/I_w$ is a Frobenius category whose associated
stable category is $2$-Calabi-Yau; moreover, it contains
the cluster-tilting object
\[
T=\Lambda/I_2\oplus \Lambda/I_2I_1\oplus
\Lambda/I_2I_1I_2\oplus \Lambda/I_2I_1I_2I_3 \oplus
\Lambda/I_w.
\]
According to Proposition~II.1.11 of [loc.~cit.], in this
decomposition, each direct factor differs from the preceding one by
one indecomposable direct summand $T_i$, $1\leq i\leq 5$, and among
these, exactly $T_3$, $T_4$ and $T_5$ are projective-injective.
Moreover, by Theorem~II.4.1 of [loc.~cit.], the quiver of the
cluster-tilting object is
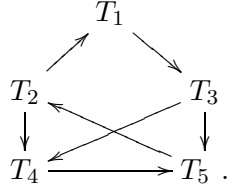
\[\xymatrix@=0.5cm{& T_1\ar[dr] \\
T_2\ar[ur]\ar[d] && T_3\ar[d]\ar[dll]\\
T_4\ar[rr]&&T_5\ar[ull] \ .}
\]
Using Theorem~I.1.6 of  \cite{BuanIyamaReitenScott07}, one can easily show that the
category $\mbox{Sub}\,\Lambda/I_w$ is a Frobenius 2-Calabi-Yau
realization of the cluster algebra $\ca(\tilde{B})$ given
by the matrix
\begin{eqnarray*}
 \tilde{B}=\left(\begin{array}{cc} 0&-1\\1&0\\-1&0\\0&-1\\0&1
        \end{array}
\right).
\end{eqnarray*}
\end{example}

We return to the general setup. Following \cite{DehyKeller08} we define a
cluster-tilting object $T'$ of $\ce$ to be {\em reachable from $T$}
if it is obtained from $T$ by a finite sequence of mutations. We
define an indecomposable rigid objects $M$ to be {\em reachable from
$T$} if it occurs as a direct factor of a cluster-tilting object
reachable from $T$.

\begin{theorem}\label{thm2}
Let $1<r\leq n$ be integers and $\ca(\tilde{B})$ the cluster
algebras with coefficients associated with an initial $n\times
r$-matrix $\tilde{B}$ of maximal rank. Suppose that $\ca(\tilde{B})$
admits a Frobenius $2$-CY realization $\ce$ with cluster tilting
object $T$.
\begin{itemize}
\item [a)] The map $M\mapsto X_M'$ induces a bijection from the set
of isomorphism classes of indecomposable rigid non projective
objects of $\ce$ reachable from $T$ onto the set of cluster
variables of $\ca(\tilde{B})$. Under this bijection, the
cluster-tilting objects reachable from $T$ correspond to the
clusters of $\ca(\tilde{B})$.
\item [b)] The map $M\mapsto \mbox{ind}_T(M)$ is a bijection from
the set of isomorphism classes of indecomposable rigid non
projective objects of $\ce$ reachable from $T$ onto the set of
$\mathbf{g}$-vectors of cluster variables of $\ca(\tilde{B})$.
\end{itemize}
\end{theorem}
\begin{proof}
 a) It follows from Theorem~\ref{thm1} c) that $X_M'$ is a cluster variable
 for each indecomposable rigid $M$ reachable from $T$ and from the existence
 of a cluster structure on $\ce$ that the map $M\mapsto X_M'$
 is a surjection onto the set of cluster variables. The injectivity of the map
 $M\mapsto X_M'$ follows from Lemma~\ref{lem2} and Proposition~\ref{prop3}.
 The second statement follows from the first one and the fact that
 $\ce$ has a cluster structure.
 b) The map is injective by Lemma~\ref{lem2}. It is surjective thanks
 to part a)  and Proposition~\ref{prop3}.
\end{proof}

\begin{theorem}\label{prop5}
Let $1<r\leq n$ be integers and $\ca(\tilde{B})$ the cluster
algebras with coefficients associated with an initial $n\times
r$-matrix $\tilde{B}$ of maximal rank. Suppose that $\ca(\tilde{B})$
admits a Frobenius $2$-CY realization $\ce$ with cluster tilting
object $T$.
\begin{itemize}
\item [a)]Conjecture 7.2 of \cite{FominZelevinsky07} holds for $\ca$,
  {\em i.e.} cluster monomials are linearly independent.
\item [b)]Conjecture 7.10 of \cite{FominZelevinsky07} holds for $\ca$, {\em i.e.}  \\
  1) Different cluster monomials have different $\mathbf{g}$-vectors with respect to a given initial seed.\\
  2) The $\mathbf{g}$-vectors of the cluster variables in any given
  cluster form a $\mathbb{Z}$-basis of the lattice $\mathbb{Z}^r$.
\item [c)] Conjecture 7.12 of \cite{FominZelevinsky07} holds for
  $\ca$, {\em i.e.} if $(g_1,\ldots, g_r)$ and $(g_1',\ldots,g_r')$
  are the $\mathbf{g}$-vectors of one and the same cluster variable
  with respect to two clusters $t$ and $t'$ related by the mutation at
  $l$, then we have
\begin{eqnarray*}
  g_j'=\left\{\begin{array}{ll} -g_l\ \ \ \ \ \ \ \ \ \ \ \ \ \ \, \quad   \ \ \ \ \ \quad \quad \quad \mbox{if}\ \, j=l\\
      g_j+[b_{jl}]_+g_l-b_{jl}\min(g_l,0)\ \ \ \mbox{if}\ \, j\neq l
             \end{array}
\right.
\end{eqnarray*}
where the $b_{ij}$ are the entries of the $r\times r$-matrix $B$
associated with $t$ and we write $[x]_+$ for $\mbox{max}(x,0)$ for any
integer $x$.
\end{itemize}
\end{theorem}
\begin{proof} a) By Theorem~\ref{thm2} a), each cluster monomial $m$ is
the image $X'_M$ of a rigid object $M$ of $\ce$, where $M$ does not
have any non zero projective direct factor. Moreover, such an object
$M$ is unique up to isomorphism. Thus, given a set $m_1$, \ldots,
$m_N$ of pairwise distinct cluster monomials, we obtain a set $M_1$,
\ldots $M_N$ of pairwise non isomorphic rigid objects without
projective direct factors such that $X'_{M_i}=m_i$ for $1\leq i\leq
N$. Thus, by Corollary~\ref{cor1} b), the images $X'_{M_i}=m_i$ of
the $M_i$ are not only pairwise distinct but in fact linearly
independent.

b) Let us prove 1). Let $m$ and $m'$ be two distinct cluster
monomials. We would like to compare their $\mathbf{g}$-vectors with
respect to a given initial cluster. By theorem~\ref{thm2} a), we may
assume that this given cluster consists of the images under
$M\mapsto X'_M$ of the indecomposable direct factors of $T$. Still
by theorem~\ref{thm2} a), the monomials $m$ and $m'$ are the images
$X'_M$ and $X'_{M'}$ of two non isomorphic rigid objects $M$ and
$M'$ of $\ce$ without non zero projective direct factors. Thus $M$
and $M'$ are still non isomorphic in the stable category
$\cc=\ul{\ce}$. But by theorem~2.3 of \cite{DehyKeller08}, non
isomorphic rigid objects have distinct indices $\ind_T(M)$ and
$\ind_T(M')$. Therefore, they have distinct $\mathbf{g}$-vectors by
Proposition~\ref{prop3}. Now let us prove 2). Let a cluster
$\mathbf{x}'$ be given. By theorem~\ref{thm2} a), the variables
$x'_i$ in $\mathbf{x}'$ are the images under $M\mapsto X'_M$ of the
indecomposable non projective direct summands $T'_i$ of a
cluster-tilting object $T'$ reachable from $T$. By
Proposition~\ref{prop3}, the $\mathbf{g}$-vector of each $x'_i$ is
the index of $T'_i$. Now by Theorem~2.6 of \cite{DehyKeller08}, the
indices of the indecomposable direct factors of a cluster-tilting object form a
basis of the lattice $K_0(\add \ul{T})$, where $\ul{T}$ is the image
of $T$ in $\cc$. Thus the $\mathbf{g}$-vectors of the $x'_i$ form a
basis of the lattice $\Z^r$.

c) By Theorem~\ref{thm2} a), we may assume that under the maps
$M \mapsto X'_M$, the clusters $t$ and
$t'$ correspond  to the cluster-tilting object
$T$ and another cluster-tilting object $T'$ obtained from $T$ by
mutation at the non projective indecomposable direct factor $T_l$.
Moreover, the given cluster variable $x$ corresponds to some non
projective rigid indecomposable object $X$. By
Proposition~\ref{prop3}, the $\mathbf{g}$-vectors of $x$ with
respect to $t$ and $t'$ are given by the components of the indices
$\ind_T(X)$ and $\ind_{T'}(X)$ in the bases formed by the
$\ind_T(T_i)$, $1\leq i\leq r$, respectively the $\ind_{T'}(T'_i)$,
$1\leq i\leq r$, where the $T_i$ and the $T'_i$ are the non
projective indecomposable direct factors of $T$ respectively $T'$.
Now Theorem~3.1 of \cite{DehyKeller08} tells us exactly how
$\ind_T(X)$ and $\ind_{T'}(X)$ are related: Let
\[
\xymatrix{T_l \ar[r] & E' \ar[r] & T_l^* \ar[r] & \Sigma T_l}
\mbox{ and }
\xymatrix{T_l^* \ar[r] & E \ar[r] & T_l \ar[r] & \Sigma T_l^*}
\]
be the exchange triangles associated with the mutation from $T$ to $T'$.
Let
\[
\phi_+ : K_0(\add T) \to K_0 (\add T') \mbox{ and }
\phi_- : K_0(\add T) \to K_0 (\add T')
\]
be the linear maps which send the classes $[T_i]$, $i\neq l$,
to themselves and send $[T_l]$ to
\[
\phi_+([T_l]) = [E]-[T_l^*] \mbox{ respectively }
\phi_-([T_l]) = [E'] - [T_l^*].
\]
Then by Theorem~3.1 of \cite{DehyKeller08}, we have
\[
\ind_{T'}(X) = \left\{ \begin{array}{ll}
\phi_+(\ind_T(X)) & \mbox{ if } [\ind_T(X): T_l]\geq 0 \\
\phi_-(\ind_T(X)) & \mbox{ if } [\ind_T(X): T_l]\leq 0.
\end{array} \right.
\]
We leave it to the reader to check that this yields exactly
the rule given in the assertion.
\end{proof}

Let $\tilde{B}$ be a $2r\times r$ matrix whose principal ($i.e.$ top
$r\times r$) part $B_0$ is mutation equivalent to an acyclic matrix,
and whose complementary ($i.e.$ bottom) part is the $r\times r$
identity matrix. Let $\ca(\tilde{B})$ be the cluster algebra with the
initial seed $({\bf x}, \tilde{B})$.

\begin{theorem}
  With the above notation, the cluster algebra $\ca(\tilde{B})$ does
  not admit a Frobenius $2$-CY realization.
\end{theorem}

\begin{proof}
  Suppose that $\ca(\tilde{B})$ has a Frobenius $2$-CY realization
  $\ce$.  Then there is a cluster tilting object $T$ of $\ce$ with
  $2r$ indecomposable direct summands.  Then we have $B(T)^0 =
  \tilde{B}$. Since $B_0$ is mutation equivalent to an acyclic matrix
  $B_c$ by a series of mutations, we have a cluster tilting object
  $T'$ such that the quiver of the stable endomorphism algebra of $T'$
  corresponds to $B_c$. Let $A$ be the stable endomorphism algebra of
  $T'$.  By the main theorem of \cite{KellerReiten06}, we have a triangle
  equivalence $\ul{\ce}\simeq \cc_{A}$, where $\cc_{A}$ is the cluster
  category of $ A$. Thus the cluster tilting graph of $\ce$ is
  connected and every rigid object of $\ce$ can be extended to a
  cluster tilting object of $\ce$.

  Let $F = \Hom_{\ce}(T,?)$. Let $S_i$, $1\leq i\leq 2r$, be the
  simple modules of $\End_{\ce}(T)$. For each object $M$ of $\ce$, we
  have the Laurent polynomial
\[
 X_M'=\prod_{i=1}^{2r}x_i^{\langle FM, S_i\rangle_{\tau}} \sum_e \chi(Gr_e (\Ext^1_{\ce}(T,M))) \prod_{i=1}^{2r} x_i^{- \langle e, S_i\rangle_3}.
\]
 Let
\[
 y_j=\prod_{i=1}^{2r} x_i^{b_{ij}},  1\leq j\leq r.
\]
As in Proposition~\ref{prop3}, we can rewrite $X_M'$ as
\[
 X_M' =  \prod_{i=1}^{2r}x_i^{\langle FM, S_i\rangle_{\tau}} (1+ \sum_{e\neq 0} \chi(Gr_e(\Ext_{\ce}^1(T,M))) \prod_{i=1}^r y_j^{e_j}),
\]
where $e_j$ is the $j$-th coordinate of $e$ in the basis of the $S_i$, $1\leq i\leq 2r$.
If the indecomposable object $M$ is rigid and not isomorphic to $T_i$ for $r< i\leq 2r$, then $X_M'$ is a  cluster variable of $\ca(\tilde{B})$. By the definition of the rational function $\mathcal{F}_{l,t}$ associated with the cluster variable $x_{l,t}$ in \cite{FominZelevinsky07},
 we have
\begin{eqnarray*}
 \mathcal{F}_M &=&  X_M'(x_1=x_2=\ldots =x_r=1)\\
&=& \prod_{i=r+1}^{2r}x_i^{\langle FM, S_i\rangle_{\tau}} (1+\sum_{e\neq 0} \chi(Gr_e(\Ext_{\ce}^1(T,M))) \prod_{j=r+1}^{2r} x_j^{e_{j-r}}).
\end{eqnarray*}
Put
\[
 G_M= 1+\sum_{e\neq 0} \chi(Gr_e\Ext_{\ce}^1(T,M)) \prod_{j=r+1}^{2r} x_j^{e_{j-r}}.
\]
Note that $G_M$ is always a polynomial of $x_i$, $r+1\leq i\leq 2r$, with constant term 1.
By Proposition 5.2 in \cite{FominZelevinsky07}, we know that the polynomial $\mathcal{F}_M$ is not divisible by $x_i$, $r+1\leq i\leq  2r$.
Now for $i> r$, we have $\langle FM, S_i\rangle_{\tau}\geq 0$ in general,  which implies that $\langle FM, S_i\rangle_{\tau}=0$.
In particular, $\langle FM, S_i\rangle_{\tau} =  [\mbox{ind}_T(M) : T_i] = 0$, for $r+1\leq i \leq 2r$.
Consider $M=\Sigma T_1$, which is rigid and indecomposable, so $X_M'$ is a cluster variable of the cluster algebra $\ca(\tilde{B})$. But in the Frobenius category $\ce$ we have the conflation
\[
 0\to T_1\to P\to \Sigma T_1\to 0,
\]
where $P$ is an injective hull of $T_1$,
which implies
\[
 \mbox{ind}_T(M)=[P]-[T_1].
\]
Thus there is always some $r+1 \leq i\leq 2r$ such that  $[\mbox{ind}_T(M) : T_i] \neq  0$. Contradiction.
\end{proof}

\begin{remark} In the above notation, if $B_0$ is acyclic, then it is easy to
deduce that the cluster algebra $\ca(\tilde{B})$ does not have a
Frobenius $2$-CY realization. Indeed in this case, one of the first
$r$  vertices of $Q$  which corresponds to $\tilde{B}$ is always a
sink. This is incompatible with the existence of a Frobenius $2$-CY
realization by remark~\ref{remark:No2CYFrobeniusRealization}.
\end{remark}

\section{Triangulated 2-Calabi-Yau realizations}
\label{T2R}

\subsection{Definitions.}
Let $B=(b_{ij})_{n\times n}$ be an antisymmetric integer matrix and
$\ca(B)$ the associated cluster algebra. A $2$-Calabi-Yau
triangulated category $\cc$ is called a {\em triangulated
2-Calabi-Yau realization} of the matrix $B$ if $\cc$ admits a
cluster tilting object $T$ such that
\begin{itemize}
 \item $\cc$ has a cluster structure in the sense
 \cite{BuanIyamaReitenScott07}, \confer~section~\ref{ss:CalabiYauCategories};
\item $T$ has exactly $n$ non isomorphic indecomposable direct summands  $T_1$, $\ldots$, $T_n$;
\item The antisymmetric matrix $B(T)$ associated with the quiver of the
endomorphism algebra of $T$ equals $B$.
\end{itemize}
We denote a triangulated $2$-CY realization of $B$ by $\cc\supset
\add T$.

Let $n_1$ and $n_2$ be positive integers. Let $B_1$ and $B_2$ be
antisymmetric integer $n_1\times n_1$ resp. $n_2\times
n_2$-matrices. Let $B_{21}$ be an integer $n_2\times n_1$-matrix
with non negative entries. Let $\cc_i\supset \ct_i$ be a
triangulated $2$-CY realization of $B_i$, $i=1,2$. Let $B$ be the
matrix
\begin{eqnarray*}
 \left(\begin{array}{ll} B_1&-B_{21}^t\\B_{21}&\quad B_2
        \end{array}
\right).
\end{eqnarray*}

A {\em gluing} of $\cc_1\supset \ct_1$ with $\cc_2\supset \ct_2$
with respect to $B$ is  a triangulated $2$-CY realization
$\cc\supset \ct$ of $B$ endowed with full additive subcategories
$\ct_1'$ and $\ct_2'$ such that
\begin{itemize}
 \item $\Hom_{\cc}(\ct_1',\ct_2')=0$;
\item The set $\mbox{indec}(\ct)$ is the disjoint
union of $\mbox{indec}(\ct_1')$ with $\mbox{indec}(\ct_2')$;
\item There is a  triangle equivalence
\[
 ^{\perp}(\Sigma \ct_1')/(\ct_1')\stackrel{\sim}{\longrightarrow} \cc_2
\]
inducing an equivalence $\ct_2'\stackrel{\sim}{\longrightarrow}\ct_2$;
\item There is a triangle equivalence
\[
 ^{\perp}(\Sigma \ct_2')/(\ct_2')\stackrel{\sim}{\longrightarrow}\cc_1
\]
inducing an equivalence $\ct_1'\stackrel{\sim}{\longrightarrow} \ct_1$.
\end{itemize}
A {\em principal gluing} of $\cc_1\supset \ct_1$ is a gluing  of
$\cc_1\supset \ct_1$ with $\cc_2\supset \ct_2$ with respect to
\begin{eqnarray*}
 \left(\begin{array}{ll} B_1&-I_{n_1} \\ I_{n_1}&\quad 0 \end{array} \right),
\end{eqnarray*}
where $\cc_2$ is the cluster category of $(A_1)^{n_1}$ and $\ct_2$
the image of the subcategory of  finitely generated projective modules.

It is well known that each acyclic matrix $B$ admits a triangulated
$2$-CY realization $\cc_{Q_B}$, where $\cc_{Q_B}$ is the cluster
category of the quiver $Q_B$ corresponding to $B$. In the last
subsection, we will see that $\cc_{Q_B}$ does admit a principal
gluing.
\begin{conjecture} \label{conj:ExistenceOfGlueings}
If $\cc_1$ and $\cc_2$ are algebraic, a gluing exists for any matrix
$B_{21}$ with non negative entries.
\end{conjecture}
Amiot's work \cite{Amiot08a} provides some evidence for the
conjecture: Indeed, if $\cc_1$ and $\cc_2$ are generalized cluster
categories \cite{Amiot08a} associated with Jacobi-finite quivers
with potential \cite{DerksenWeymanZelevinsky07}, it is easy to
construct a quiver with potential which provides a gluing as
required by the conjecture.

\subsection{Cluster algebras with coefficients.}
Let $B$ be an antisymmetric integer $n\times n$-matrix. Suppose that
the matrix $B$ admits a triangulated $2$-CY realization $\cc$ with
the
 cluster tilting subcategory $\ct=\add T$. Let $T_i$, $1\leq i\leq
n$, be the non isomorphic indecomposable direct summands of $T$. By
the definition, we have $B(T)=B$. The mutations of the  matrix $B$
correspond to the mutations of the cluster tilting object $T$. Fix
an integer $0< r\leq n$ and consider the submatrix $B^0$ of $B$
formed by the first $r$ columns of $B$. If $l\leq r$, then we have
\[
 \mu_l(B^0)=(\mu_l(B))^0,
\]
where $\mu_l$ is the mutation in the direction $l$.  Thus we can
view the cluster algebra $\ca(B^0)$ with coefficients as a
sub-cluster algebra of $\ca(B)$, \confer  Ch.~III  of \cite{BuanIyamaReitenScott07}.

Denote  by $\cp$ the full subcategory of $\cc$ whose objects are the
finite direct sums of copies of $T_{r+1},\ldots, T_n$.   We define a
subcategory of $\cc$
\[
 \cu=\,^\perp\!(\Sigma\cp)=\{X\in \cc| \Ext^1_{\cc}(T_i,X)=0\mbox{  for  }r<i\leq n\}.
\]
By Theorem~I.2.1 of \cite{BuanIyamaReitenScott07}, the quotient
category $\cu/\cp$ is a 2-Calabi-Yau triangulated category and the
projection $\cu \to \cu/\cp$ induces a bijection between the cluster
tilting subcategories of $\cc$ containing $\cp$ and the cluster
tilting subcategories of $\cu/\cp$. Thus, a mutation of a cluster
tilting object in $\cu/\cp$ can be viewed as a mutation of a cluster
tilting object in $\cu\subset\cc$ which does not affect the direct
summands $T_i$, $r<i\leq n$. This exactly corresponds to a mutation
of the matrix $B$ in one of the first $r$ directions. In particular,
a mutation of the cluster algebra $\ca(B^0)$ corresponds to a
mutation of a cluster tilting object in $\cu$.

Recall from section~\ref{ss:ClusterCharacters} that on $\cc$, we
have Palu's cluster character associated with $T$, which is given by
the formula
\[
X_M=X^T_M=\prod_{i=1}^{n}x_i^{-[\operatorname{coind}_T M: T_i]}
\sum_e\chi(Gr_e(\Hom_{\cc}(T, M)))\prod_{i=1}^{n}x_i^{\langle S_i,
e\rangle_a}.
\]
We consider the composition of this map with the shift:
\[
X_M'=X_{\Sigma M} =\prod_{i=1}^{n}x_i^{[\operatorname{ind}_T M: T_i]}
\sum_e\chi(Gr_e(\Hom_{\cc}(T, \Sigma M)))\prod_{i=1}^{n}x_i^{\langle
  S_i, e\rangle_a}.
\]
We consider the restriction of the map $M \mapsto X_M' $ to the
subcategory $\cu$. It follows from
Proposition~\ref{prop:ClusterCharacterProperties} that if $M$ is an
indecomposable rigid object reachable from $T$ in $\cu$, then $X_M'$
is a cluster variable of $\ca(B^0)$. We will rewrite this variable
so as to express its $\mathbf{g}$-vector (if it is defined) in terms
of the index of $M$: Let $M$ be an object of $\cu$. Then
$\Hom_{\cc}(T, \Sigma M)$ is an $\End_{\cc}(T)$-module which
vanishes at each vertex $r< i\leq n$. Let $e$ be the image
of $\Hom_{\cc}(T, \Sigma M)$ in the
Grothendieck group of $\mod \End_{\cc}(T)$. Let $e_j$ be the $j$-th
coordinate of $e$ with respect to the basis $S_i$, $1\leq i\leq n$.
We have
\begin{eqnarray*}
  \langle S_i, e\rangle_a&=&\langle S_i, e\rangle_{\tau} -\langle e, S_i\rangle_{\tau}\\
  &=& \sum_{j=1}^r e_j(\langle S_i, S_j\rangle_{\tau} - \langle S_j, S_i\rangle_{\tau}) \\
  &=& \sum_{j=1}^r e_j (\Ext^1_{\End_{\cc}(T)}(S_j,S_i)-\Ext^1_{\End_{\cc}(T)}(S_i,S_j))\\
  &=& \sum_{j=1}^r b_{ij}e_j.
\end{eqnarray*}
As in section~\ref{IG}, put
\[
 y_j=\prod_{i=1}^{n}x_i^{b_{ij}}\, , \mbox{for} \  1\leq j\leq r.
\]
Then $X_M'$ can be rewritten as
\[
X_M'=\prod_{i=1}^{n}x_i^{[\operatorname{ind}_T ( M):
T_i]}(1+\sum_{e\neq 0} \chi(Gr_e(\Hom_{\cc}(T, \Sigma
M)))\prod_{j=1}^ry_j^{e_j}).
\]
As in section~\ref{IG}, when $\rank\, B^0=r$, we can define the
$\mathbf{g}$-vector of $M\in \cu$ with respect to a cluster tilting
object $T$. Thus we have proved part a) of the following
proposition. We leave the easy proof of part b) to the reader.
\begin{proposition}\label{prop6} Suppose that $\rank\,B^0=r$. Let
$M$ be an object of $\cu$.
\begin{itemize}
\item[a)] The $\mathbf{g}$-vector of $X_M'$ with respect to the
initial cluster is given by
\[
g_T(X_M')(i)= [\mbox{ind}_T\, (M): T_i],\,\mbox{for}\,  1\leq i\leq r.
\]
\item[b)] The index of the image of $M$ in $\cu/\cp$ with respect to
the image of $T$ is
\[
\sum_{i=1}^r g_T(X_M')(i)[T_i].
\]
\end{itemize}
\end{proposition}
In analogy with the definition in section~\ref{F2R}, we define a
cluster tilting object $T'$ of $\cu$ to be {\em reachable} from $T$
if it is obtained from $T$ by a sequence of mutations at
indecomposable rigid objects of $\cu$ not in $\cp$. We define an
indecomposable rigid object of $\cu$ to be {\em reachable} from $T$
if it is a direct factor of a cluster tilting object reachable from
$T$.

\begin{theorem}\label{thm3}
Let $B$ be an antisymmetric integer $n\times n$-matrix and $1\leq
r\leq n$ an integer such that the submatrix $B^0$ of $B$ formed by
the first $r$ columns has rank $r$. Let $\ca=\ca(B^0)$ be the
associated cluster algebra with coefficients. Assume that the matrix
$B$ admits a triangulated $2$-CY realization given by a triangulated
category $\cc$ with a cluster tilting object $T$ which is the sum of
$n$ indecomposable direct factors $T_1$, \ldots, $T_n$. Denote  by
$\cp$ the full subcategory of $\cc$ whose objects are the finite
direct sums of copies of $T_{r+1},\ldots, T_n$ and define the
subcategory $\cu$ of $\cc$ by
\[
 \cu=\,^\perp\!(\Sigma\cp)=\{X\in \cc| \Ext^1_{\cc}(T_i,X)=0\mbox{ for } r<i\leq n\}.
\]
For $M\in\cc$, define (\confer section~\ref{ss:ClusterCharacters})
\[
X_M'=X^T_{\Sigma M} =\prod_{i=1}^{n}x_i^{[\operatorname{ind}_T M:
T_i]} \sum_e\chi(Gr_e(\Hom_{\cc}(T, \Sigma
M)))\prod_{i=1}^{n}x_i^{\langle
  S_i, e\rangle_a}.
\]
Then the following hold.
\begin{itemize}
\item[a)] The map $M \mapsto X'_M$ induces a bijection from the set
of isomorphism classes of indecomposable rigid objects of $\cu$ not
belonging to $\cp$ and reachable from $T$ onto the set of cluster
variables of $\ca(B^0)$. Under this bijection, the cluster-tilting
objects of $\cu$ reachable from $T$ correspond to the clusters of
$\ca(B^0)$.

\item[b)] The map $M \mapsto [\ind_T(M):T_i]_{1\leq i \leq r}$ is
a bijection from the set of indecomposable rigid objects of $\cu$
not belonging to $\cp$ and reachable from $T$ onto the set of
$\mathbf{g}$-vectors of cluster variables of $\ca(B^0)$.

\item [c)]Conjecture 7.2 of \cite{FominZelevinsky07} holds for $\ca$,
  {\em i.e.} the cluster monomials are linearly independent over $\Z$. Moreover,
  the cluster monomials form a basis of the $\Z[x_{r+1}, \ldots, x_n]$-submodule
  of $\ca(B^0)$ which they generate.

\item [d)]Conjecture 7.10 of \cite{FominZelevinsky07} holds for $\ca$, {\em i.e.}  \\
  1) Different cluster monomials have different $\mathbf{g}$-vectors with respect to a given initial seed.\\
  2) The $\mathbf{g}$-vectors of the cluster variables in any given
  cluster form a $\mathbb{Z}$-basis of the lattice $\mathbb{Z}^r$.
\item [e)] Conjecture 7.12 of \cite{FominZelevinsky07} holds for
  $\ca$, {\em i.e.} if $(g_1,\ldots, g_r)$ and $(g_1',\ldots,g_r')$
  are the $\mathbf{g}$-vectors of one and the same cluster variable
  with respect to two clusters $t$ and $t'$ related by the mutation at
  $l$, then we have
\begin{eqnarray*}
  g_j'=\left\{\begin{array}{ll} -g_l\ \ \ \ \ \ \ \ \ \ \ \ \ \ \, \quad   \ \ \ \ \ \quad \quad \quad \mbox{if}\ \, j=l\\
      g_j+[b_{jl}]_+g_l-b_{jl}\min(g_l,0)\ \ \ \mbox{if}\ \, j\neq l
             \end{array}
\right.
\end{eqnarray*}
where the $b_{ij}$ are the entries of the $r\times r$-matrix $B$
associated with $t$ and we write $[x]_+$ for $\mbox{max}(x,0)$ for
any integer $x$.
\end{itemize}
\end{theorem}

\begin{proof}
It follows from Proposition~\ref{prop:ClusterCharacterProperties}
that the map $M\mapsto X'_M$ is well-defined and surjective onto the
set of cluster variables of $\ca(B^0)$. It is injective by
Proposition \ref{prop6} b) because rigid objects of
$\cu/\cp$ are determined by their indices and the map taking a rigid object $M$
of $\cu$ without non zero direct factors in $\cp$ to its image
in $\cu/\cp$ is injective (up to isomorphism). This also implies part
b). The same proof as for Corollary \ref{cor1} b) yields the linear
independence of the cluster monomials in c). Let us prove that
the cluster monomials form a basis of the $\Z[x_{r+1}, \ldots, x_n]$-submodule
of $\ca(B^0)$ which they generate. Indeed, over $\Z$, this
submodule is spanned by the images $X'_M$ of all rigid objects
of $\cu$ obtained as direct sums of objects of $\cp$ and indecomposable rigid objects
reachable from $T$ not belonging to $\cp$. Such objects $M$ are in
particular rigid in $\ct$ and they can be distinguished (up to
isomorphism) by their indices. Now again, the same proof
as for Corollary~\ref{cor1} b) shows that these $X'_M$ are linearly
independent over $\Z$. Clearly this implies that the cluster monomials
form a basis of the $\Z[x_{r+1}, \ldots, x_n]$-submodule
of $\ca(B^0)$ which they generate. As in the proof of Theorem~\ref{prop5} b),
the assertions in part d) follow from the interpretation of the $\mathbf{g}$-vector
given in \ref{prop6} b) and the facts that
\begin{itemize}
\item[1)] rigid objects of $\cu/\cp$ are determined by their indices
(Theorem~2.3 of \cite{DehyKeller08}) and
\item[2)] the indices of the indecomposable direct factors of a
cluster-tilting subcategory $\ct$ of $\cu/\cp$ form a basis of
$K_0(\ct)$ (Theorem~2.6 of \cite{DehyKeller08}).
\end{itemize}
Part e) is proved in exactly the same way as the corresponding
statement for cluster algebras with a $2$-CY Frobenius realization
in  Theorem \ref{thm2} c).
\end{proof}

\begin{example}
Let $A_4$ be the quiver $3\to 1\to 2\leftarrow 4$, $\cc_{Q}$ the
corresponding cluster category. The following is the AR quiver of
$\cc_Q$, where $P_i$, $1\leq i\leq 4$, are the indecomposable
projective $kQ$-modules.
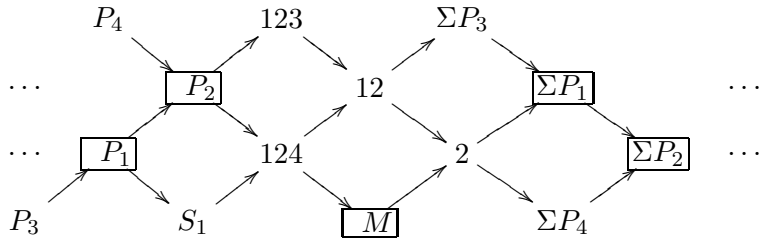
\[\xymatrix@=0.4cm{
  & P_4\ar[dr] &  & 123\ar[dr]& & \Sigma P_3\ar[dr]\\
  \ldots & & *+<0.1cm>[F]{\ \, P_2}\ar[ur]\ar[dr]& & 12\ar[ur]\ar[dr]& &  *+<0.1cm>[F]{\Sigma P_1}\ar[dr] &&\ldots\\
  \ldots & *+<0.1cm>[F]{\ \, P_1}\ar[ur]\ar[dr] & & 124\ar[ur]\ar[dr] & & 2\ar[ur]\ar[dr] & &*+<0.1cm>[F]{\Sigma P_2}&\ldots\\
  P_3\ar[ur] & & S_1\ar[ur]& & *+<0.1cm>[F]{\ \, M}\ar[ur]& & \Sigma
  P_4\ar[ur]}
\]
Let $T=P_1\oplus P_2\oplus P_3\oplus P_4$ be the canonical cluster
tilting object in $\cc_Q$, $\cp=\add (P_3\oplus P_4)$. It is easy to
see that the indecomposable objects in $\cu/\cp\cong \cc_{A_2}$ are
exactly $P_1, P_2, M, \Sigma P_1, \Sigma P_2$. In this case, the
matrix $B(T)^0$ is
\begin{eqnarray*}
 \left(\begin{array}{ll} 0&1\\-1&0\\1&0\\0&1
        \end{array}
\right).
\end{eqnarray*}
We have $\mbox{rank}\,B(T)^0=2$. Moreover,  the cluster algebra $\ca(B(T)^0)$ has principal coefficients.
\end{example}

\subsection{Cluster algebras with principal coefficients.}
In this subsection, we suppose  that $2r=n$ and  that the
complementary part of $B^0$ is the $r\times r$ identity matrix. Thus
the cluster algebra $\ca(B^0)$ has principal coefficients. Recall
that for the matrix $B$, we have a triangulated $2$-CY realization
$\cc\supset \add T$ and we have fixed $\cp=\add (T_{r+1}\oplus
\ldots\oplus T_{2r})$. Let $\cq=\add(T_1\oplus\ldots\oplus T_r)$.
Let $\cc_1=\cu/\cp$ and $\cc_2=\,^\perp(\Sigma\cq)/\cq$ be the quotient
categories, $\ct_1=\add(\pi_1(T_1\oplus\ldots\oplus T_r))$ and
$\ct_2=\add(\pi_2(T_{r+1}\oplus \ldots\oplus T_{2r}))$ the
corresponding cluster tilting subcategories, where $\pi_1$ and
$\pi_2$ are the respective projection functors. Then $\cc$ is a
gluing of $\cc_1\supset \ct_1$ with $\cc_2\supset \ct_2$ with
respect to the matrix $B$.

As in section~\ref{F2R}, for a cluster variable $x_{l,t}$ of the
cluster algebra $\ca(B^0)$ which corresponds to an indecomposable
rigid object $M\in \cu$ and  not in $\cp$, we denote the rational
function $\mathcal{F}_{l,t}$ defined in section 3 of
\cite{FominZelevinsky07} by $\mathcal{F}_M$. Since $x_{l,t}=X_M'$,
we have
\begin{eqnarray*}
 \mathcal{F}_M&=&X_M'(x_1=\ldots=x_r=1)\\
&=& \prod_{i=r+1}^{2r} x_i^{[\operatorname{ind}_T(M):T_i]}(1+\sum_{e\neq 0} \chi(Gr_e(\Hom_{\cc}(T, \Sigma M)))\prod_{j=r+1}^{2r}x_j^{e_{j-r}}).
\end{eqnarray*}
The following result is now a consequence of Proposition 3.6 and 5.2
in \cite{FominZelevinsky07}. We give a proof based on representation
theory. Note that conjecture 5.4 of \cite{FominZelevinsky07} will be
proved in full generality in \cite{DerksenWeymanZelevinsky09}.

\begin{theorem}\label{thm4}
Conjecture 5.4 of \cite{FominZelevinsky07} holds for $\ca(B^0)$, {\em i.e.}
 the polynomial $\mathcal{F}_M$ has constant term 1. Thus we have
\[
 \mathcal{F}_M=1+\sum_{e\neq 0} \chi(Gr_e(\Hom_{\cc}(T, \Sigma M)))\prod_{j=r+1}^{2r}x_j^{e_{j-r}}.
\]

\end{theorem}
\begin{proof}
We need to show that for each $i>r$, $[\mbox{ind}_T(M):T_i]$ is zero.
Since $X_M'$ is a cluster variable and  $M$ is indecomposable,
we have the following two cases:\\
{\em Case 1:} $M\cong \Sigma T_j$ for some $j\leq r$.
We have $\mbox{ind}_T(M)=-[T_j]$, which implies that
$[\mbox{ind}_T(M):T_i]=0$.\\
{\em Case 2:} $M$ is not isomorphic to $\Sigma T_j$ for any $j\leq
r$. Recall that by assumption, $M$ is not isomorphic to $T_j$ for
any $j>r$. We have the following minimal triangle
\[
 T_M^1\to T_M^0\to M\to \Sigma T_M^1
\]
with $T_M^0$, $T_M^1$ in $\add T$ and
$\mbox{ind}_T(M)=[T_M^0]-[T_M^1]$. Since $M$ belongs to $\cu$, for
each $i>r$ we have $\Hom_{\cc}(M, \Sigma T_i)=0$. If we had $[T_M^1:
T_i]\neq 0$ for some $i>r$, then the above minimal triangle would
have a non zero direct factor
\[
 T_i\to T_i\to 0\to \Sigma T_i.
\]
Suppose that we have  $[T_M^0: T_i]\neq 0$ for some $i>r$.  Applying
the functor $F=\Hom_{\cc}(T,?)$ to the triangle, we get a minimal
projective resolution of $FM$ as an $\End_{\cc}(T)$-module. Note
that for $i>r$, the projective module $FT_i$ is also a  simple
module, which implies that $FM$ is decomposable. Contradiction.
\end{proof}

Suppose that the indecomposable rigid object $M$ of $\cc$
is reachable from $T$ and consider the polynomial
$\mathcal{F}_M$ of Theorem~\ref{thm4}. We define the {\em
$f$-vector} $f_T(M)=(f_1, \ldots, f_r)$ of $M$ with respect to $T$
by
\[
 \mathcal{F}_M|_{Trop(u_1,\ldots,u_r)}(u_1^{-1},\ldots,u_r^{-1})=u_1^{-f_1}\ldots u_r^{-f_r} \ko
\]
where $\mbox{Trop}(u_1, \ldots, u_r)$ is the tropical
semifield defined in section~\ref{ss:ClusterAlgebras}.

\begin{proposition}\label{prop7}
Suppose that $M$ is  not isomorphic to $ T_i$ for $1\leq i\leq 2r$,
and let  $\dimv\, \Hom_{\cc}(T, \Sigma M)=(d_1, \ldots, d_r)$. Then
we have
\[
d_i=f_i, \,\, 1\leq i\leq r.
\]
\end{proposition}

\begin{proof}
By Theorem~\ref{thm4}, we have
\[
 \mathcal{F}_M
= 1+\sum_{e\neq 0} \chi(Gr_e(\Hom_{\cc}(T, \Sigma M)))\prod_{j=r+1}^{2r}x_j^{e_{j-r}}.
\]
Therefore, we obtain
\begin{eqnarray*}
\mathcal{ F}_M|_{Trop(u_1,\ldots, u_r)}(u_1^{-1},\ldots,
u_r^{-1})&=&1 \oplus \,
\bigoplus_{e\neq 0} \chi(Gr_e(\Hom_{\cc}(T, \Sigma M)))\prod_{j=1}^{r}u_j^{-e_j}\\
&=& u_1^{-d_1}\ldots u_n^{-d_r}.
\end{eqnarray*}
\end{proof}

Under the assumptions above, we have proved that the dimension
vector of $\Hom_{\cc}(T, \Sigma M)$ equals the $f$-vector $f_T(M)$.
Conjecture 7.17 of \cite{FominZelevinsky07} states that the
$f$-vectors coincide with the denominator vectors in general. But by
recent work of A. Buan, R. Marsh  and I. Reiten
\cite{BuanMarshReiten08a},  the dimension vectors do not always
coincide with the denominator vectors.  In fact, as shown in
\cite{BuanMarshReiten08a}, for a quiver $Q$ whose underlying graph
is an affine Dynkin diagram, the vector $\dim \Hom_{\cc_Q}(T,M)$ is
different from the denominator vector of $X^T_M$  if $M=R$ and $R$
is a direct factor of $T$, where $R$ is a rigid regular
indecomposable of maximal quasi-length. This leads to the following
minimal counterexample to Conjecture 7.17 in
\cite{FominZelevinsky07}. Let us point out that the corresponding
computations already appear in \cite{Cerulli08}. In subsection~{5.5}
below, we will show that in many cases, the $f$-vector is greater or
equal to the denominator vector.

\subsection{A counterexample.}
\begin{example}
Let $Q$ be the following quiver
\[
 \xymatrix@=0.4cm{& 3\ar[dl]\\ 1\ar@2[rr]& & 2\ar[ul]\ .}
\]
Let $\ca(Q)$ be the cluster algebra associated with the initial seed
given by $Q$ and  ${\bf x}=(x_1,x_2,x_3)$.  Consider the mutations
at $3,2,1$. Let ${\bf x}^{t_3}$ be the corresponding cluster. We have
\[
 x_1^{t_3}=\frac{x_1^2+2x_1x_2+x_2^2+x_3}{x_1x_2x_3}
\]
and the corresponding $F$-polynomial is
\[
 F_{x_1^{t_3}}=1+(1+y_1+y_1y_2)y_3+y_1y_2y_3^2.
\]
Then the $f$-vector of $x_1^{t_3}$ does not coincide with the denominator vector.
\end{example}

Let us interpret this counterexample in terms of  representation theory.
Let $A_{2,1}$ be the quiver
\[\xymatrix@=0.4cm{
& 3\ar[dr]\\
1\ar[ur]\ar[rr] &  &2\ .}
\]
Consider the cluster category $\cc_{A_{2,1}}$ of $kA_{2,1}$. Let
$P_i$, $1\leq i\leq 3$,  be the indecomposable projective modules
and $S_i$ the corresponding simple modules. Then
\[
 T=P_1\oplus P_2\oplus \tau S_3,
\]
is a  cluster tilting object of $\cc_{A_{2,1}}$,
where $\tau$ is the Auslander-Reiten translation functor.  The quiver $Q_T$ of $T$ looks like
\[
\xymatrix@=0.4cm{& \tau S_3\ar[dl]\\
P_1\ar@2[rr]& &P_2\ar[ul]\ .}
\]
We will show that the cluster category $\cc_{A_{2,1}}\supset \add T$
admits a principal gluing. For this, consider the following quiver
$Q_1:$
\[
 \xymatrix@=0.5cm{& 6\ar[r] &3\ar[d]\\
4\ar[r]&1\ar[r]\ar[u]&2&5\ar[l]\ .}
\]
It admits a cluster category $\cc_{Q_1}$. Let $T_{Q_1}=kQ_1$  be the
canonical cluster tilting object in $\cc_{Q_1}$. Let
$T'=\mu_3(\mu_6(T_{Q_1}))$ be the cluster tilting object obtained by
mutations from $T_{Q_1}$. Denote the non isomorphic indecomposable
direct summands of $T'$ by $T'_i$, $1\leq i\leq 6$. Then the quiver
of $Q_{T'}$ is
\[
 \xymatrix@=0.4cm{& & T'_6\ar[d]\\
& & T'_3\ar[dl]\\
T'_4\ar[r]&T'_1\ar@ 2[rr]& &T'_2\ar[ul]&T'_5\ar[l]\ .}
\]
Let $\cp= \add (T'_4\oplus T'_5\oplus T'_6)$. Then $\cu/\cp$ is a
2-Calabi-Yau triangulated category and admits a cluster tilting
object with the quiver $Q_T$. By the main theorem of
\cite{KellerReiten06}, we know that there is a triangle equivalence
$\cu/\cp\simeq \cc_{A_{2,1}}$. Thus, we see that the  matrix $B(T')$
admits a triangulated $2$-CY realization $\cc_{Q_1}$ which is the
required principal gluing of $\cc_{A_{2,1}}\supset \add T$.  We may
assume that  the images of  $T'_1, T_2', T_3'$ coincide with $P_1,
P_2, \tau S_3$ in $\cc_{A_{2,1}}$ respectively.  Denote the shift
functor in $\cc_{Q_1}$ (resp. $\cc_{A_{2,1}}$) by $\Sigma$ (resp.
$[1]$).

Let $N$ be the preimage of $S_3$ in $\cc_{Q_1}$. Then one can easily
compute
\[
 \dimv \Hom_{\cc_{Q_1}}(T',\Sigma N)=\dimv \Hom_{\cc_{A_{2,1}}}(T, \tau S_3)=(1,1,2).
\]
Note that  the denominator vector of $X_N'$ equals the denominator
vector of $X_{\tau S_3}^T$. Now the result follows from the
Proposition above.

\subsection{An inequality.}
Let $\ct$ be a 2-Calabi-Yau triangulated category with
cluster-tilting object $T$. Recall that we have the generalized
Caldero-Chapoton map
\[
 X_M^T = \prod_{i=1}^n x_i^{-[\operatorname{coind}_{T}(M): T_i]}
\sum_e \chi(Gr_e(GM)) \, \prod_{i=1}^n x_i^{\langle S_i,e\rangle_a},
\]
where $G$ is  the functor $\Hom_{\cc}(T,?): \cc\to \mod
\End_{\cc}(T)$. The following proposition is proved in greater
generality in \cite{DerksenWeymanZelevinsky09}.
\begin{proposition}\label{prop8}
 For each $M$ in $\ct$, let $\dimv\, GM=(m_1,\ldots, m_n)$ and let $1\leq i\leq n$. We have
\[
-[\mbox{coind}_T(M):T_i] + \langle S_i, e\rangle_a \geq -m_i,
\]
for each submodule $N$ of $GM$ with $\dimv N=e$. Thus the exponent
of $x_i$ in the denominator of $X_M$ is less or equal to $m_i$.
\end{proposition}
\begin{proof}
This result holds for the case $M\cong \Sigma T',\,T'\in \add T$
obviously. We assume that $M$ is indecomposable and not isomorphic
to any $\Sigma T'$.  The  case where $M$ is decomposable is a
consequence of the multiplication theorem for $X_?$. Now by Lemma~7
of \cite{Palu08}, we have $$-[\mbox{coind}_T(M):T_i]=-\langle S_i,
GM\rangle_{\tau}.$$ Note that we have the short exact sequence of
$\End_{\cc}(T)$-modules
\[
 0\to N\to GM\to GM/N\to 0.
\]
By applying the functor $\Hom(S_i, ?)$, we get
\[
 \langle S_i, N\rangle_{\tau}+\langle S_i, GM/N\rangle_{\tau}-\langle S_i, GM\rangle_{\tau}+ \mbox{dim}\, \Ext^2(S_i, N)\geq 0.
\]
By the stable 3-Calabi-Yau property of $\mod \End_{\cc}(T)$ proved
in \cite{KellerReiten07}, we have $\mbox{dim}\Ext^2(S_i,N)\leq
\mbox{dim} \Ext^1(N,S_i)$. Therefore, we have
\begin{eqnarray*}
 -[\mbox{coind}_T(M):T_i] + \langle S_i, e\rangle_a &\geq& -\langle N,S_i\rangle_{\tau}-\langle S_i, GM/N\rangle_{\tau}-\mbox{dim}\Ext^2(S_i,N)\\
&\geq&-[N,S_i]- [S_i, GM/N]+ \,^1[S_i, GM/N]\\
&\geq& - m_i.
\end{eqnarray*}
\end{proof}

\subsection{Behaviour of the $\mathbf{g}$-vectors under mutation.}
Let $B=(b_{ij})$ be an antisymmetric integer $r\times r$-matrix. Let
$\cc\supset \add T$ be a triangulated $2$-CY realization of $B$. Let
$T_1, \ldots, T_r$ be the non isomorphic indecomposable factors of
$T$. Let $1\leq l\leq r$ be an integer and $T'=\mu_l(T)$ the
mutation of $T$ at $T_l$. Thus, the non isomorphic indecomposable
factors of $T'$ are $T_1,\ldots, T_l^*,\ldots, T_r$. Let $\cc_1$ be
a principal gluing of $\cc\supset \add T$ and $\cc_2$ a principal
gluing of $\cc\supset \add T'$ (we assume such gluings exist).
For each indecomposable object
$M\in \cc$ reachable from $T$, we denote by $\mathcal{F}_M^{T}$ and
$\mathcal{F}_M^{T'}$ the $\mathcal{F}$-polynomials of $M$ with
respect to $\cc_1$ and $\cc_2$ respectively.  Following
\cite{FominZelevinsky07}, we define the integers $h_l$ and $h_l'$ by
\begin{eqnarray*}
u^{h_l}=\mathcal{F}_M^{T}|_{Trop(u)}(u^{[-b_{k1}]_{+}},\ldots, u^{-1},\ldots, u^{[-b_{kn}]_{+}}),\\
u^{h_l'}=\mathcal{F}_M^{T'}|_{Trop(u)}(u^{[b_{k1}]_{+}},\ldots,
u^{-1},\ldots, u^{[b_{kn}]_{+}}),
\end{eqnarray*}
where $u^{-1}$ is in the $l$-th position.

The following proposition shows that if the gluings $\cc_1$ and
$\cc_2$ exist (for example if $\cc$ is algebraic and
Conjecture~\ref{conj:ExistenceOfGlueings} holds), then Conjecture~6.10
of \cite{FominZelevinsky07} holds for the cluster
algebra with principal coefficients associated with  $B$.

\begin{proposition}\label{prop9}
 In the above notation, we have
\[
h_l'=-[[\mbox{ind}_{T}(M):T_l]]_{+}, \ \ h_l=\mbox{min}(0,
[\mbox{ind}_{T}(M):T_l]).
\]
\end{proposition}
\begin{proof}
Let $S_i$, $1\leq i\leq r$, be the top of the indecomposable right
projective $\End_{\cc}(T')$-module $\Hom_{\cc}(T', T'_i)$. First we
will show that $g_l=[\mbox{ind}_T(M):T_l]> 0$ iff $S_l$ occurs as a
submodule of the module $\Hom_{\cc}(T', \Sigma M)$ and that the
multiplicity of $S_l$ in the socle of  $\Hom_{\cc}(T', \Sigma M)$
equals $[\mbox{ind}_T(M):T_l]$.

Suppose  that $g_l> 0$. Then we have the following triangle
\[
 T^1_M\to T^{0'}_M\oplus (T_l)^{g_l}\to M\to \Sigma T^1_M
\]
with $T^1_M$, $T^{0'}_M$ in $\add T$ and $[T^{0'}_M:T_l]=0$, where
$(T_l)^{g_l}$ is the sum of  $g_{l}$ copies of $T_l$.  Applying the
functor $\Hom_{\cc}(T', ?)$ to the shift of the above triangle, we
get the exact sequence
\[
 0\to \Hom_{\cc}(T', \Sigma (T_l)^{g_l})\to \Hom_{\cc}(T',
 \Sigma M)\to \Hom_{\cc}(T', \Sigma^2 T^1_M) \to \ldots
\]
Note that $\Hom_{\cc}(T',\Sigma (T_l)^g)\cong(S_l)^{g_l}$,  {\em
i.e.} $S_l$ occurs  with multiplicity $\geq g_l$ in the socle of
$\Hom_{\cc}(T', \Sigma M)$. If the multiplicity of $S_l$ in the
socle of $\Hom_{\cc}(T',\Sigma M)$ was $>g_l$, then $S_l$ would
occur in the socle of  $\Hom_{\cc}(T',\Sigma^2T_M^1)$. This is not
the case since $\Hom_{\cc}(T',\Sigma^2T_M^1)$ is the sum of
injective indecomposables not isomorphic to the injective hull
$\Hom_{\cc}(T',\Sigma^2T_l)$ of $S_l$. Conversely, if $S_l$ occurs
in the socle of $\Hom_{\cc}(T', \Sigma M)$, thanks to the split
idempotents property of $\cc$, we have an irreducible morphism
$\alpha: \Sigma T_l\to \Sigma M$ in $\cc$.  Thus, by the definition
of the index, we get $g_l>0$. Moreover, the multiplicity of $S_l$
equals $g_l$ by the same argument as before.

Assume that $g_l>0$. For an arbitrary submodule  $U$ of
$\Hom_{\cc}(T', \Sigma M)$, let $\mbox{dim}\, U=(e_1, \ldots, e_n)$.
We will show that
\[
e_l\leq g_l +\sum_i[b_{il}]_{+}e_i.
\]

Indeed, consider the projective resolution of the simple module $S_l$
\[
 \ldots \to \oplus P_i^{b_{il}}\to P_l\to S_l\to 0.
\]
Applying the functor $\Hom_{\End_{\cc}(T')}(?,U)$, we get  the exact sequence
\[
 0\to \Hom(S_l, U)\to \Hom(P_l, U)\to \Hom(\oplus P_i^{b_{il}}, U)\to \ldots,
\]
which implies the inequality because the dimension of $\Hom(S_l,U)$
is less or equal to the multiplicity of $S_l$ in the socle of
$\Hom_{\cc}(T',\Sigma M)$, which equals $g_l$. By
Theorem~\ref{thm4}, we have

\begin{eqnarray*}
 u^{h_l'}&=&\mathcal{F}_M^{T'}|_{Trop(u)}(u^{[b_{k1}]_{+}},\ldots, u^{-1},\ldots, u^{[b_{kn}]_{+}})\\
&=& 1\oplus \bigoplus_{e}\chi(Gr_e(\Hom_{\cc}(T',\Sigma M)))u^{-e_l}\prod_{i\neq l}(u^{[b_{ki}]_{+}})^{e_i}.
\end{eqnarray*}
We have just shown that for each $e$, we have
\[
 -e_l+\sum_i[b_{il}]_{+}e_i\geq g_l,
\]
and the equality occurs if $e$ is the dimension vector of the
submodule $(S_l)^{g_l}$.  We conclude that we have
$h_l'=-[\mbox{ind}_{T}(M):T_l]$. If $g_l\leq 0$, then $S_l$ does not
occur in the socle of $\Hom_{\cc}(T',\Sigma M)$ and it is easy to
see that $h_l'=0$. Dually, we have the equality $h_l=\mbox{min}(0,
[\mbox{ind}_{T}(M):T_l])$.
\end{proof}

\subsection{Acyclic cluster algebras with principal coefficients.}

Let $B$ be an antisymmetric integer $r\times r$-matrix. Assume that
$B$ is acyclic. Let $Q$ be the corresponding quiver of $B$ with set
of vertices $Q_0=\{1,\ldots, r\}$ and with set of arrows $Q_1$. Let
$\cc_Q$ be the cluster category of $Q$, $T=kQ$ the canonical cluster
tilting object of $\cc_Q$. We claim that the cluster category
$\cc_Q\supset \add T$ admits a principal gluing.

Indeed, we define a new quiver $\tilde{Q} =
Q\overleftarrow{\coprod}Q_0$ associated with $Q$: Its set of
vertices is $\{1,\ldots, 2r\}$, and its arrows are those of $Q$ and
new arrows from $r+i$ to $i$ for each vertex $i$ of $Q$. Since $Q$
is acyclic,  so is $\tilde{Q}$,  hence $k\tilde{Q}$ is
finite-dimensional and  hereditary. Thus,  we have the cluster
category $\cc_{\tilde{Q}}$ which is a triangulated $2$-CY
realization of the matrix
\begin{eqnarray*}
 \left(\begin{array}{ll} B&-I_r\\I_r&\quad 0
        \end{array}
\right).
\end{eqnarray*}
In particular, $\cc_{\tilde{Q}}\supset \add k\tilde{Q}$ is a
principal gluing for $\cc_Q\supset \add T$. Thus,
Proposition~\ref{prop6}, Theorem~\ref{thm3}, Theorem~\ref{thm4} and
Proposition~\ref{prop7} hold for acyclic cluster algebras with
principal coefficients.

Let $P_i$, $1\leq i\leq 2r$, be the non isomorphic indecomposable
projective right  modules of $k\tilde{Q}$.  Let $\cp=\add
(P_{r+1}\oplus \ldots \oplus P_{2r})$.  We have a triangle
equivalence
\[
 ^\perp(\Sigma\cp)/\cp\stackrel{\sim}{\longrightarrow} \cc_Q.
\]

 Recall that there is a partial order on $\Z^r$ defined by
\[
\alpha \leq \beta  \,\,\mbox{iff}\ \, \alpha(i)\leq \beta(i), \
\mbox{for} \, 1\leq i\leq r, \, \mbox{where} \, \alpha,\, \beta \in
\Z^r.
\]

\begin{proposition}
Let $B$ be a $2r\times r$   integer matrix, whose principal part is
antisymmetric and acyclic and whose complementary part is the
identity matrix. Let $\sigma$ be a sequence $k_1$,  $\ldots$, $k_m$
with $1\leq k_i\leq r$. Denote by $B_{\sigma}$ the matrix
\[
 \mu_{k_1}\circ\mu_{k_2}\ldots \circ\mu_{k_m}(B)=(b_{ij}^{\sigma}).
\]
Let $E_{\sigma}=(e_1,e_2,\ldots, e_r)$ be the complementary part of
$B_{\sigma}$, where $e_i\in \Z^r$, $1\leq i\leq r$.  Then for each
$i$, we have $e_i\leq 0$ or $e_i\geq 0$.
\end{proposition}
\begin{proof}
Suppose that there is some $k$ such that $e_k\nleq 0$ and $e_k\ngeq
0$. For simplicity, assume that $k=1$, {\em i.e.} there are $\ r<
i,\,j \leq 2r$ such that $b_{i1}^{\sigma}>0$ and
$b_{j1}^{\sigma}<0$.

Let $Q$ be the quiver corresponding to the principal part of $B$ and
$\tilde{Q}$ as constructed above. By the argument above, there is a
cluster tilting object $T'$ of $\cc_{k\tilde{Q}}$ such that
$B(T')^0= B_{\sigma}$. We have arrows $P_i\to T_1'$ and $T_1'\to
P_j$, where $T_1'$ is the indecomposable direct summand of $T'$
corresponding to the first column of $B_{\sigma}$. Now if we
consider the mutation in direction $1$ of $T'$, we will have an
arrow $P_i\to P_j$ in $Q_{\mu_1(T')}$. But this is impossible, since
for $r< l\leq 2r$, the $P_l$ are simple pairwise non isomorphic
modules so we have
\[
\Hom_{\cc_{k\tilde{Q}}}(P_i,P_j) =  \Hom_{k\tilde{Q}}(P_i, P_j) =0.
\]
\end{proof}



\def\cprime{$'$}
\providecommand{\bysame}{\leavevmode\hbox to3em{\hrulefill}\thinspace}
\providecommand{\MR}{\relax\ifhmode\unskip\space\fi MR }
\providecommand{\MRhref}[2]{%
  \href{http://www.ams.org/mathscinet-getitem?mr=#1}{#2}
}
\providecommand{\href}[2]{#2}

\end{document}